\documentclass[a4paper,12pt]{article}
\usepackage{geometry}

\usepackage[T1]{fontenc}
\usepackage{lmodern}

\usepackage{amsmath}
\usepackage{amssymb}
\usepackage{amsthm}

\usepackage{mathrsfs}

\usepackage{braket}
\usepackage{graphicx}

\usepackage{booktabs}

\usepackage{enumerate}

\usepackage[abbrev, shortalphabetic]{amsrefs}

\newcommand{\R}{\mathbb{R}}
\newcommand{\Z}{\mathbb{Z}}
\newcommand{\N}{\mathbb{N}}
\newcommand{\K}{\mathcal{K}}

\newcommand{\Span}{\operatorname{span}}

\newcommand{\conv}{\operatorname{conv}}
\newcommand{\pos}{\operatorname{pos}}

\newcommand{\interior}{\operatorname{int}}

\usepackage{bm}

\newcommand{\va}{\bm{a}}
\newcommand{\vb}{\bm{b}}
\newcommand{\vc}{\bm{c}}

\newcommand{\vn}{\bm{n}}
\newcommand{\vp}{\bm{p}}
\newcommand{\vr}{\bm{r}}

\newcommand{\vu}{\bm{u}}
\newcommand{\vv}{\bm{v}}
\newcommand{\vw}{\bm{w}}
\newcommand{\vx}{\bm{x}}
\newcommand{\vy}{\bm{y}}

\newcommand{\checkK}{\check{\mathcal{K}}}

\theoremstyle{plain}
\newtheorem{theorem}{Theorem}[section]
\newtheorem{lemma}[theorem]{Lemma}
\newtheorem{proposition}[theorem]{Proposition}
\newtheorem{corollary}[theorem]{Corollary}

\theoremstyle{definition}
\newtheorem{remark}[theorem]{Remark}

\newtheorem{conjecture}[theorem]{Conjecture}

\newcommand{\abs}[1]{\lvert#1\rvert}
\newcommand{\norm}[1]{\lVert#1\rVert}

\usepackage{multirow}

\usepackage{ulem}

\usepackage{tikz}
\usetikzlibrary{cd}
\usetikzlibrary{intersections}
\usetikzlibrary{calc}

\usepackage{wasysym}
\newcommand{\Pentagon}{\pentagon}

\newcommand{\Marc}[1]{\stackrel{\raisebox{-1pt}{\rotatebox{90}{\big)}}}{#1}}

\begin{document}

\title{Minimal volume product of three dimensional convex bodies invariant under certain groups of order four}

\author{Hiroshi Iriyeh
\thanks{
College of Science, 
Ibaraki University, 
e-mail: hiroshi.irie.math@vc.ibaraki.ac.jp}
\and 
Masataka Shibata
\thanks{
Department of Mathematics, Meijo University, 
e-mail: mshibata@meijo-u.ac.jp
}}

\maketitle
%
%
%
%
%
%
%

%
%
%
%
%
%
\begin{abstract}
We give the sharp lower bound of the volume product of three dimensional convex bodies which are invariant under two kinds of 
discrete subgroups of $O(3)$ of order four.
We also characterize the convex bodies with the minimal volume product in each case.
This provides new partial results of the non-symmetric version of Mahler's conjecture in the three dimensional case.
\end{abstract}
%
%


\section{Introduction and main results}

%
%
A compact convex set in $\R^n$ with nonempty interior is called a {\it convex body}.
A convex body $K \subset \R^n$ is said to be {\it centrally symmetric} if it satisfies that $K=-K$.
We write $\mathcal{K}^n$ for the set of all convex bodies in $\R^n$ equipped with the Hausdorff metric and by $\mathcal{K}^n_0$ the set of all $K \in \mathcal{K}^n$ which are centrally symmetric.
The interior of $K \in \mathcal{K}^n$ is denoted by $\mathrm{int}(K)$.
The {\it polar body} of $K$ with respect to $z \in \mathrm{int}(K)$ is defined by
\begin{equation*}
K^z :=
\left\{
y \in \R^n;
(y-z) \cdot (x-z) \leq 1 \text{ for any } x \in K
\right\},
\end{equation*}
where $\cdot$ denotes the standard inner product on $\R^n$.
And $\abs{K}$ denotes the $n$-dimensional Lebesgue measure (volume) of $K$ in $\R^n$.
Then the {\it volume product} of $K$ is defined by
\begin{equation}
\label{eq:a}
 \mathcal{P}(K) := \min_{z \in \mathrm{int}(K)}\abs{K} \, \abs{K^z},
\end{equation}
which is invariant under non-singular affine transformations of $\R^n$.
It is known that for each $K \in \mathcal{K}^n$ the minimum of \eqref{eq:a} is attained at the unique point on $\mathrm{int}(K)$, which is called the {\it Santal\'o point} of $K$ (see e.g., \cite{MP}).
For a centrally symmetric convex body $K \in \mathcal{K}^n_0$, the Santal\'o point of $K$ is the origin $o$.
In this case, the right-hand side of \eqref{eq:a} is nothing but $\abs{K} \, \abs{K^\circ}$, where $K^\circ$ denotes the polar body of $K$ with respect to $o$.

Mahler conjecture \cite{Ma1} is a longstanding open problem which states that $\mathcal{P}(K)=\abs{K} \, \abs{K^\circ} \geq 4^n/n!$ for any $K \in \mathcal{K}^n_0$.
This conjecture was confirmed for $n=2$ \cite{Ma2} and $n=3$ \cite{IS}, but the case that $n\geq4$ is still open.
For other partial results, see, e.g., references in \cite{Sc}*{Section 10.7} and \cite{IS}.
Note that an $n$-dimensional cube or, more generally, Hanner polytopes attain the value $4^n/n!$, and these polytopes are conjectured as the minimizers of $\mathcal{P}$.

The following is also known as the non-symmetric version of Mahler conjecture, which does not impose the symmetric assumption on $K$.

\begin{conjecture}
\label{conj}
Any $K \in \mathcal{K}^n$
satisfies that
\begin{equation*}
\mathcal{P}(K) \geq \abs{\triangle^n} \abs{(\triangle^n)^\circ}=\frac{(n+1)^{n+1}}{(n!)^2},
\end{equation*}
where $\triangle^n$ denotes a regular $n$-simplex with $o$ as the centroid. 
\end{conjecture}

Conjecture \ref{conj} was proved by Mahler \cite{Ma2} for $n=2$ (see also \cite{Me} for the equality case), but remains open for $n \geq 3$.

A centrally symmetric convex body is regarded as the one which are invariant under the group $\braket{-E} \cong \Z_2$, where $E$ is the identity matrix of size $n$.
Then, for a discrete subgroup $G$ of $O(n)$, it is natural to consider the minimization problem of the volume product $\mathcal{P}(K)$ for $G$-invariant convex bodies $K \in \mathcal{K}^n$ (see e.g., Problem in \cite{IS2}*{p.~739}).
The first result along this line was obtained by \cite{BF}.
To explain it we need some more preparation.
We write $\mathcal{K}^n(G)$ for the set of $G$-invariant convex bodies in $\R^n$.
For a convex body $K_0 \in \mathcal{K}^n$, we consider the subgroup of $SO(n)$ or $O(n)$ defined by
\begin{equation*}
 SO(K_0):=\left\{ g \in SO(n); gK_0=K_0 \right\}, \ \ 
 O(K_0):=\left\{ g \in O(n); gK_0=K_0 \right\}.
\end{equation*}
For instance, as groups $O(\triangle^n)$ and $SO(\triangle^n)$ are isomorphic to the symmetric group and
the alternating group of degree $n+1$, respectively.
For convex bodies in $\mathcal{K}^n(O(\triangle^n))$,
Conjecture \ref{conj} was proved by F.~Barthe and M.~Fradelizi \cite{BF}*{Theorem 1 (i)}.
After that, in \cite{IS3}*{Theorem 1.4}, we relaxed the condition and proved Conjecture \ref{conj} for $K \in \mathcal{K}^n(SO(\triangle^n))$ by using the method of signed volume estimate (see \cite{IS2}*{Theorem 1.5 (i)} for $n=3$).
We remark that the advantage of this method have been explained in \cite{IS2}*{Section 1.4}.

In this paper, we focus on the three dimensional case, and we shall further relax the condition of \cite{IS2}*{Theorem 1.5 (i)} and provide new partial solutions of Conjecture \ref{conj}, based on not only the signed volume estimate but also the equipartition of which effectiveness will be explained in the last part of this section.
In order to give the precise statement, we briefly recall the list of all discrete subgroups of $O(3)$ by using Schoenflies' notation.

We equip $\R^3$ with the usual orthogonal $xyz$-axies and put
\begin{equation*}
R_\ell:=
\begin{pmatrix}
 \cos(2\pi/\ell) & - \sin(2\pi/\ell) & 0 \\
 \sin(2\pi/\ell) & \cos(2\pi/\ell) & 0 \\
 0 & 0 & 1
\end{pmatrix}, 
V:=
\begin{pmatrix}
 1 &  0 & 0 \\
 0 & -1 & 0 \\
 0 &  0 & 1
\end{pmatrix}, 
H:=
\begin{pmatrix}
 1 & 0 &  0 \\
 0 & 1 &  0 \\
 0 & 0 & -1
\end{pmatrix},
\end{equation*}
where $\ell \in \N$. 
Here, $R_\ell$ means the rotation through the angle $2\pi/\ell$ about the $z$-axis.
It is known that, up to conjugation, the discrete subgroups of $O(3)$ are classified as the seven infinite families
\begin{equation*}
\begin{aligned}
& C_\ell:=\braket{R_\ell}, \quad
C_{\ell h}:=\braket{R_\ell,H}, \quad
C_{\ell v}:=\braket{R_\ell,V}, \quad
S_{2\ell}:=\braket{R_{2\ell}H}, \\
& D_\ell:=\braket{R_\ell,VH}, \quad
D_{\ell d}:=\braket{R_{2\ell}H,V}, \quad
D_{\ell h}:=\braket{R_\ell,V,H},
\end{aligned}
\end{equation*}
and the seven finite groups as follows:
\begin{equation*}
\begin{aligned}
T&:=\left\{g \in SO(3); g \triangle =\triangle \right\}, &
T_d&:=\left\{g \in O(3); g \triangle =\triangle \right\}, \quad
T_h:= \left\{\pm g ; g \in T\right\}, \\
O&:=\left\{g \in SO(3); g \Diamond =\Diamond \right\}, &
O_h&:=\left\{g \in O(3); g \Diamond =\Diamond \right\}= \left\{\pm g ; g \in O\right\}, \\
I&:=\left\{g \in SO(3); g \Pentagon =\Pentagon \right\}, &
I_h&:=\left\{g \in O(3); g \Pentagon =\Pentagon \right\}= \left\{\pm g ; g \in I\right\},
\end{aligned}
\end{equation*}
where $\triangle:=\triangle^3$, $\Diamond:=\Diamond^3$, and $\Pentagon:=\Pentagon^3$ denote the regular tetrahedron (regular $3$-simplex), the regular octahedron, and the regular icosahedron with $o$ as their centroids, respectively (see e.g., \cite{CS}*{Table 3.1}).
In general, if the fixed point set of the action of a group $G$ on $\R^3$ consists only on the origin $o$, then both the Santal\'o point and the centroid of $K \in \mathcal{K}^3(G)$ coincide with $o$.
Let $G$ be one of the group in the above list other than $C_\ell$, $C_{\ell v}$ $(\ell \in \N)$.
Then, $G$ has this property, and so $\mathcal{P}(K)=\abs{K}\, \abs{K^\circ}$ holds for $K \in \mathcal{K}^3(G)$, which is similar to the centrally symmetric case.

Now, we focus on the subgroups of $T_d$ or $T$ in the following table:
\begin{equation*}
\begin{tikzcd}[row sep=-0.5em, column sep=0.5em]
& D_2 \arrow[r, phantom, sloped, "\subset"] & ~T~ \arrow[rd, phantom, sloped, "\!\!\!\subset"] \\
{\{e\}} \arrow[ru, phantom, sloped, "\subset"] \arrow[rd, phantom, sloped, "\subset"] & & & T_d \\
& S_4 \arrow[r, phantom, sloped, "\subset"] & D_{2d} \arrow[ru, phantom, sloped, "\!\!\!\subset"]
\end{tikzcd}
\end{equation*}
The orders of the groups $D_2$, $S_4$, $D_{2d}$, $T$, and $T_d$ are $4$, $4$, $8$, $12$, and $24$, respectively.
As groups, $D_2$, $S_4$, and $D_{2d}$ are isomorphic to $\Z_2 \times \Z_2$, $\Z_4$, and the Dihedral group of order $8$, repsectively.
The following are our main results, which provide new partial solutions of Conjecture \ref{conj}.

\begin{theorem}
\label{thm:1}
For any $K \in \mathcal{K}^3(D_2)$, it holds that
\begin{equation*}
\abs{K} \, \abs{K^\circ} \geq \abs{\triangle}\,\abs{\triangle^\circ} = \frac{64}{9},
\end{equation*}
with equality if and only if $K$ is a linear image of $\triangle$ or $\triangle^\circ$ which is invariant under the group $D_2$.
\end{theorem}

\begin{theorem}
\label{thm:2}
For any $K \in \mathcal{K}^3(S_4)$, it holds that
\begin{equation*}
\abs{K} \, \abs{K^\circ} \geq \abs{\triangle}\,\abs{\triangle^\circ} = \frac{64}{9},
\end{equation*}
with equality if and only if $K$ is a linear image of $\triangle$ or $\triangle^\circ$ which is invariant under the group $S_4$.
\end{theorem}

Since $S_4$ is a subgroup of $D_{2d}$ and $\triangle \in \mathcal{K}^3(D_{2d})$, Theorem \ref{thm:2} yields the following result.

\begin{corollary}
\label{thm:3}
For any $K \in \mathcal{K}^3(D_{2d})$, it holds that
$\abs{K} \, \abs{K^\circ} \geq \abs{\triangle}\,\abs{\triangle^\circ} = 64/9$,
with equality if and only if $K$ is a linear image of $\triangle$ or $\triangle^\circ$ which is invariant under the group $D_{2d}$.
\end{corollary}

In this way, it seems a natural approach to Conjecture \ref{conj} to relax gradually the assumption of the discrete subgroup $G \subset O(3)$, especially, its order.

\begin{remark}
\begin{enumerate}[(i)]
\item 
Theorems \ref{thm:1} and \ref{thm:2} extend the results for $G=T$ (\cite{IS2}*{Theorem~1.5 (i)}) and for $G=T_d$ in \cite{BF}*{Theorem~1 (i)}, respectively.
\item 
Theorems \ref{thm:1} and \ref{thm:2}, and Corollary \ref{thm:3} are also solutions of Problem in \cite{IS2}*{p.~739} for the groups $D_2$, $S_4$, and $D_{2d}$, respectively.
In the three dimensional case, the problem is still open for $G=D_{\ell d}$, $S_{2 \ell}$ $(\ell \geq 4)$ in the case where the fixed point set of the action of $G$ on $\R^3$ is $\{o\}$, and otherwise for $C_\ell$, $C_{\ell v}$ $(\ell \geq 1)$.
\end{enumerate}
\end{remark}
Other than the results explained until now,
partial results concerning Conjecture \ref{conj} are not so many.
In \cite{BF}, besides \cite{BF}*{Theorem 1 (i)} explained above, they proved that Conjecture \ref{conj} is true for convex bodies in $\R^n$ which have reflection symmetries with respect to hyperplanes $H_1, \dots, H_m \subset \R^n$ with the property that $\bigcap_{i=1}^m H_i=\{o\}$.
M.~Meyer and S.~Reisner \cite{MR} confirmed the conjecture for polytopes in $\R^n$ possessing at most $n+3$ vertices.
In the three dimensional case, Conjecture \ref{conj} is true in the case where a convex body $K \subset \R^3$ satisfies
$K=\operatorname{conv}(K \cap H_1,K \cap H_2)$ for two planes $H_1, H_2 \subset \R^3$ (\cite{FMZ}).
And the conjecture is also true for the convex bodies in $\R^n$ which are close to the $n$-simplex with $o$ as the centroid with respect to the Banach--Mazur distance (\cite{KR}).

This paper is organized as follows.
In Section \ref{sec:2}, we first review necessary facts about two and three dimensional bodies.
Known facts concerning the signed volume estimate are arranged from Lemma \ref{lem:5} to Lemma \ref{lem:1}, which have been used in \cite{IS} and \cite{IS2}.
However, for the characterization of the equality case in Theorems \ref{thm:1} and \ref{thm:2}, we need further notions and results.
We provide a technique to handle a two dimensional ``cone'' whose central angle is larger than $\pi$ (Lemmas \ref{lem:8} and \ref{lem:3}).
By using it, we can extend the result \cite{IS2}*{Proposition 3.6}, which was a basic tool to characterize the equality case for $K \in \mathcal{K}^3(T)$ (Lemma \ref{lem:4}).
We explain and prove these results from scratch.

In Section \ref{sec:3}, we prove Theorem \ref{thm:1}.
For $K \in \mathcal{K}^3(D_2)$, we can choose a fundamental domain of the action of $D_2$ on $K$ with its three boundary sections.
However, their sizes are a priori different from one another.
On the other hand, as we treated in \cite{IS2}*{Section 4.1}, $K \in \mathcal{K}^3(SO(\triangle^n)) \cong \mathcal{K}^3(T)$ is equipped with additional symmetries, so that the three boundary sections above are congruent one another, which enabled us to use the signed volume estimate directly for $K \in \mathcal{K}^3(T)$.
This is the essential difference between $\mathcal{K}^3(D_2)$ and $\mathcal{K}^3(T)$, and is the reason that we have to do equipartition of $K \in \mathcal{K}^3(D_2)$.
In Section \ref{sec:3.3}, by using the freedom of the action of nonsingular linear transformations (commute with $D_2$), we prove the existence of the good division of $K$ which satisfies the condition \eqref{eq:1}.
Owing to \eqref{eq:1}, we can carry out the the signed volume estimate to get the inequality in Theorem \ref{thm:1}.
In Section \ref{sec:3.4}, we characterize the equality case of the inequality by using new tools prepared in Section \ref{sec:2}.

In Section \ref{sec:4}, we prove Theorem \ref{thm:2} for $K \in \mathcal{K}^3(S_4)$.
Although the story of the proof is basically the same as in Section \ref{sec:3}, the calculation is more complicated than the case $K \in  \mathcal{K}^3(D_2)$, because the three boundary sections of the fundamental domain are not the same type in contrast to the case 
of $D_2$.
After the preparation in Section \ref{sec:4.1}, we establish the inequality of Theorem \ref{thm:2} in Section \ref{sec:4.2}.
Finally, in Section \ref{sec:4.3}, we characterize the equality case of the inequality.

\section{Preliminaries}
\label{sec:2}

From now on, we write $\mathcal{K}^3$ as $\mathcal{K}$.
Let $K \in \mathcal{K}$ be a convex body whose interior contains the origin $o$.
For any $g \in O(3)$, we have
\begin{equation*}
 (g K)^\circ = (g^{\mathrm{T}})^{-1} K^\circ = g K^\circ.
\end{equation*}
Hence, for each subgroup $G \subset O(3)$, $K \in \mathcal{K}(G)$ means that  $K^\circ \in \mathcal{K}(G)$.

Let $\mathcal C$ be an oriented piecewise $C^\infty$-curve in $\R^3$ and $\vr(t)$ $(0 \leq t \leq 1)$ a parametrization of $\mathcal C$.
Then we define a vector $\overline{\mathcal{C}} \in \R^3$ by
\begin{equation*}
\overline{\mathcal{C}} := \frac{1}{2} \int_\mathcal{C} \vr \times d \vr = \frac{1}{2} \int_0^1 \vr(t) \times \vr'(t) \,dt.
\end{equation*}
Note that $\overline{\mathcal{C}}$ is independent of the choice of a parametrization of $\mathcal C$.
If the curve $\mathcal C$ is on a plane in $\R^3$ passing through the origin $o$, then $\overline{\mathcal C}$ is a normal vector of the plane.
For any $g \in O(3)$, we obtain
\begin{equation*}
\overline{g \mathcal{C}} = \det(g) g \,\overline{\mathcal C}.
\end{equation*}
We denote by $\mathcal{C}^{-1}$ the curve in $\R^3$ with the same image of the curve $\mathcal C$ but with the opposite direction
(see \cite{IS3}*{p.~18006}).
Then $\overline{\mathcal{C}^{-1}}=-\overline{\mathcal C}$ holds.

We consider the ruled surface
$o*{\mathcal C}:= \{ \lambda \vu \in \R^3; \vu \in {\mathcal C}, 0 \leq \lambda \leq 1 \}.$
For any $\vx \in \R^3$, the {\it signed volume} of the solid
\begin{equation*}
\vx*(o*{\mathcal C})
:=\{ (1-\nu)\vx+\nu \vy \in \R^3; \vy \in o*{\mathcal C}, 0 \leq \nu \leq 1 \}
\end{equation*}
is defined by
\begin{equation*}
 \frac{1}{3} \vx \cdot \overline{\mathcal C}
= \frac{1}{6} \int_0^1 \det(\vx, \vr(t), \vr'(t)) \,dt.
\end{equation*}

Let $K \in \mathcal{K}$ with $o \in \interior K$.
The {\it radial function} of $K$ is defined by $\rho_K(\vx) := \max\{ \lambda \geq 0; \lambda \vx \in K \}$ for $\vx \in \R^3 \setminus \{ o \}$.
For any two vectors $\va, \vb \in \R^3 \setminus \{o\}$
with $\va \nparallel \vb$, let us consider the oriented curve $\mathcal{C}(\va,\vb):=\mathcal{C}_K(\va,\vb)$ on the boundary $\partial K$ defined by
\begin{equation*}
\mathcal{C}_K(\va,\vb):=\{ \rho_K((1-t)\va + t\vb) ((1-t)\va+t\vb) \in \R^3;  0 \leq t \leq 1 \}.
\end{equation*}
For any vectors $\va_1,\dots,\va_m \in \R^3 \setminus \{o\}$
 with $\va_i \nparallel \va_{i+1}$, let $\mathcal{C}_K(\va_1,\dots,\va_m)$ denote the oriented curve on $\partial K$ consisting of the successive oriented curves $\mathcal{C}(\va_i,\va_{i+1})$, $i=1,\ldots,m-1$;
\begin{equation*}
\mathcal{C}_K(\va_1,\dots,\va_m):=\mathcal{C}(\va_1,\va_2) \cup \cdots \cup \mathcal{C}(\va_{m-1},\va_m).
\end{equation*}
If $\mathcal{C}_K(\va_1,\dots,\va_m)$ is a simple closed curve on $\partial K$, that is, 
it has no self-intersection and $\rho_K(\va_m)\va_m=\rho_K(\va_1)\va_1$, then we denote by $\mathcal{S}_K(\va_1,\dots,\va_{m-1})$ the part of $\partial K$ enclosed by the curve such that the orientation of the part is compatible with that of the curve.

Let $K \in \K$ with $o \in \interior K$ and $\mu_K$ denotes its Minkowski gauge.
We note that $\rho_K(\vx)=1/\mu_K(\vx)$ for $\vx \in \R^3 \setminus \{o\}$ and $K=\{ \vx \in \R^3; \mu_K(\vx) \leq 1 \}$.
Suppose that $\partial K$ is a $C^\infty$-hypersurface in $\R^3$, then $K$ is said to be {\it strongly convex} if the Hessian matrix $D^2(\mu_K^2/2)(\vx)$ of $C^\infty$-function $\mu_K^2/2$ on $\R^3$ is positive definite for each $\vx \in \R^3$ with $\norm{\vx}=1$, where $\norm{\vx} :=\sqrt{\vx \cdot \vx}$.
To prove the inequalities in Theorems \ref{thm:1} and \ref{thm:2}, we introduce the following class of convex bodies:
\begin{equation*}
\checkK:=\left\{
K \in \mathcal{K}; o \in \interior K, \text{$K$ is strongly convex, $\partial K$ is of class $C^\infty$}
\right\}.
\end{equation*}
For a convex body $K \in \checkK$, we can define a $C^\infty$-map 
$\Lambda=\Lambda_K: \partial K \to \partial K^\circ$ by
\begin{equation*}
\Lambda(\vx)=\nabla \mu_K(\vx) \quad (\vx \in \partial K).
\end{equation*}
If $K$ is strongly convex with the boundary $\partial K$ of class $C^\infty$, then $K^\circ$ is also strongly convex with $\partial K^\circ$ of class $C^\infty$, and the map $\Lambda: \partial K \to \partial K^\circ$ is a $C^\infty$-diffeomorphism satisfying that $\vx \cdot \Lambda(\vx)=1$ (see \cite{Sc}*{Section 1.7.2} and \cite{IS}*{Section 3.1}).
As for a two dimensional convex body $K \subset \R^2$, we can also define the map $\Lambda=\Lambda_K$ in a similar way.
Note that the curve $\mathcal{C}(\va,\vb)$ on $\partial K$ is in the plane of $\R^3$ spanned by $\va$ and $\vb$, but its image $\Lambda(\mathcal{C}(\va,\vb))$ into $\partial K^\circ$ is not necessarily contained in a plane of $\R^3$.

\begin{lemma}[\cite{IS2}*{Lemma 3.1}]
\label{lem:5}
Let $K \in \checkK$ be a convex body and $\va, \vb \in \R^3 \setminus \{o\}$ be two vectors with $\va \nparallel \vb$.
Let $H \subset \R^3$ be the plane spanned by $\va$ and $\vb$, and $\Pi_H$ denotes the orthogonal projection from $\R^3$ onto $H$.
Then $\Pi_H \circ \Lambda_K = \Lambda_{K \cap H}$ holds on $\partial(K \cap H)$.
\end{lemma}

We move on to the case of $G$-invariant convex bodies in $\checkK$.
For a subgroup $G$ of $O(3)$,
we introduce the class
\begin{equation*}
\checkK(G):= \left\{K \in \checkK; gK=K \text{ for all } g \in G\right\}.
\end{equation*}
Then we obtain the following formula.

\begin{lemma}[\cite{IS2}*{Lemma 3.2}]
\label{lem:6}
Let $K \in \checkK(G)$.
For any $\vx \in \partial K$ and $g \in G$, we obtain
\begin{equation*}
g \Lambda_K(\vx)
= \Lambda_K(g \vx)
\end{equation*}
\end{lemma}

Although not every $G$-invariant convex body $K$ is equipped with the map $\Lambda_K$,
due to the following approximation result, it is sufficient to consider the class $\checkK(G)$ for estimating $\abs{K} \, \abs{K^\circ}$ from below.

\begin{proposition}[Schneider \cite{Sc2}*{p.~438}]
\label{prop:5}
Let $G$ be a discrete subgroup of $O(3)$.
Let $K \in \K(G)$ be a $G$-invariant convex body.
Then, for any $\varepsilon > 0$ there exists a $G$-invariant convex body $K_\epsilon \in \checkK(G)$ having the property that $\delta(K,K_\epsilon) < \varepsilon$, where $\delta$ denotes the Hausdorff distance on $\K$.
\end{proposition}

In addition, to estimale $\abs{K}\, \abs{K^\circ}$ from below sharply we need the following two lemmas.
\begin{lemma}[\cite{IS2}*{Proposition 3.5}]
\label{lem:2}
Let $K \in \checkK$.
For any $\va, \vb \in \partial K$ with $\va \nparallel \vb$,
it holds that
\begin{equation*}
 \overline{\mathcal{C}(\va,\vb)} \cdot \overline{\Lambda(\mathcal{C}(\va,\vb))} \geq \frac{1}{4}(\va-\vb) \cdot (\Lambda(\va) - \Lambda(\vb)).
\end{equation*}
\end{lemma}

\begin{lemma}[\cite{IS2}*{Lemma 3.10}]
\label{lem:1}
Let $K \in \checkK$.
Let $\mathcal C$ be a piecewise $C^\infty$, oriented, simple closed curve on $\partial K$.
Let $\mathcal{S}_K(\mathcal{C}) \subset \partial K$ be a piece of surface enclosed by the curve $\mathcal C$ such that the orientation of the surface is compatible with that of $\mathcal C$.
Then, the following inequality holds:
\begin{equation*}
\abs{o*\mathcal{S}_K(\mathcal{C})} \, \abs{o*\mathcal{S}_{K^\circ}(\Lambda(\mathcal{C}))}
\geq \frac{1}{9} \, \overline{\mathcal{C}} \cdot \overline{\Lambda(\mathcal{C})}.
\end{equation*} 
\end{lemma}
From now on, we make preparations for characterizing the equality conditions of the inequalities in Theorems \ref{thm:1} and \ref{thm:2}.
We first recall an estimate of the volume product for a part of a plane convex body $K \subset \R^2$.
For our purpose, we have to extend the formula given by \cite{BMMR}*{Lemma 7} as follows
(see also \cite{IS2}*{Lemmas 2.1, 2.3 and 2.4}).

\begin{lemma}
\label{lem:8}
Let $K \subset \R^2$ be a convex body containing the origin $o$ in its interior.
Assume that $\va, \vb \in \partial K$ with $\va \neq \vb$ and $\va^\circ, \vb^\circ \in \partial K^\circ$ with $\va^\circ \neq \vb^\circ$.
Here, we measure the angles $\angle \va o \vb$ and $\angle \va^\circ o \vb^\circ$ counterclockwise.
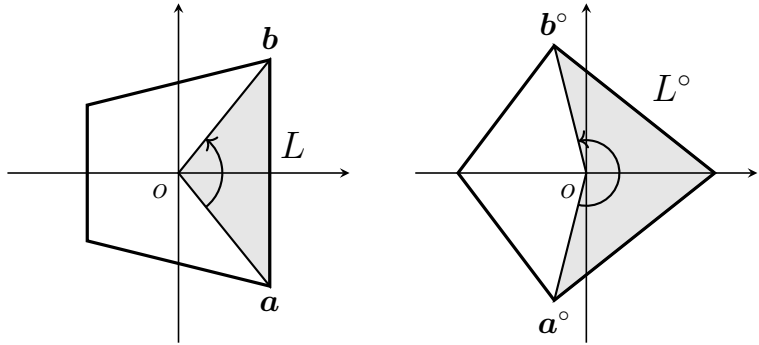
\begin{center}
\begin{tabular}{ccc}
\begin{tikzpicture}[scale=1.5,baseline=0]
\coordinate[label=below left:$o$] (O) at (0,0);

\def\scale{0.8}

\coordinate (A) at ($\scale*(1,5/4)$);
\coordinate (B) at ($\scale*(-1,3/4)$);
\coordinate (C) at ($\scale*(-1,-3/4)$);
\coordinate (D) at ($\scale*(1,-5/4)$);
\path[name path=line 1] (O)--(1,2);
\path[name path=line 2] (A)--(B);
\path[name path=line 3] (O)--(1,-2);
\path[name path=line 4] (C)--(D);
\path[name  intersections={of=line 1  and line 2,  by={E}}];
\path[name  intersections={of=line 3  and line 4,  by={F}}];

\filldraw[fill=lightgray!40!white, thick] (D)--(O)--(A)--cycle; 
\draw[very thick] (A)--(B)--(C)--(D)--cycle;
\draw[->,>=stealth,semithick] (-1.5,0)--(1.5,0); 
\draw[->,>=stealth,semithick] (0,-1.5)--(0,1.5); 
\node at (A) [above] {$\vb$};
\node at (D) [below] {$\va$};
\node at (20/20,5/20) {\large$L$};

\draw [->, thick] ($0.3*\scale*(1,-5/4)$) arc [start angle = -51.34019, end angle = 51.34019, radius = {\scale*0.3*1.60078}];
\end{tikzpicture}
&  & 
\begin{tikzpicture}[scale=1.5, baseline=0]
\coordinate[label=below left:$o$] (O) at (0,0);

\def\scale{1.125}

\coordinate (A) at ($\scale*(1,0)$);
\coordinate (B) at ($\scale*(-1/4,1)$);
\coordinate (C) at ($\scale*(-1,0)$);
\coordinate (D) at ($\scale*(-1/4,-1)$);
\filldraw[fill=lightgray!40!white, thick] (O)--(D)--(A)--(B)--cycle; 
\draw[very  thick] (A)--(B)--(C)--(D)--cycle; 
\draw[->,>=stealth,semithick] (-1.5,0)--(1.5,0); 
\draw[->,>=stealth,semithick] (0,-1.5)--(0,1.5); 
\node at (B) [above] {$\vb^\circ$};
\node at (D) [below] {$\va^\circ$};
\node at (15/20,15/20) {\large$L^\circ$};

\draw [->, thick] ($\scale*0.25*(-1/4,-1)$) arc [start angle = -104.04, end angle = 104.04, radius = {\scale*0.25*1.0307764}];
\end{tikzpicture}
\\
\end{tabular}
\end{center}
We put $L:= o\,* \! \Marc{\va \vb} \left(=o*\mathcal{C}_K(\va, \vb) \text{ if } 0<\angle \va o \vb < \pi \right)$ and $L^\circ:= o\,* \Marc{\va^\circ \vb^\circ\!\!}\,\,$.
Then we have
\begin{equation}
\label{2Dineq}
\abs{L} \, \abs{L^\circ} \geq \frac{1}{4}(\va-\vb)\cdot (\va^\circ-\vb^\circ).
\end{equation}
If the equality of \eqref{2Dineq} holds, then there exist two points $\vc \in \partial K$ and $\vc^\circ \in \partial K^\circ$ such that 
\begin{equation*}
L=o*([\va, \vc] \cup [\vc, \vb]), \quad \vc \cdot J\va \geq 0, \quad \vc \cdot (-J\vb) \geq 0
\end{equation*}
and
\begin{equation*}
L^\circ=o*([\va^\circ, \vc^\circ] \cup [\vc^\circ, \vb^\circ]), \quad \vc^\circ \cdot J\va^\circ \geq 0, \quad \vc^\circ \cdot (-J\vb^\circ) \geq 0,
\end{equation*}
where J denotes  the standard complex structure on $\R^2$.
\end{lemma}

\begin{proof}
For any $\vx, \vy \in \R^2$, we have
\begin{equation*}
\det(\vx,\vy) = \vx \cdot (-J\vy).
\end{equation*}
For any point $\vx \in K$, the sum of the signed area of the triangle $o \va \vx$ and that of $o \vx \vb$ is less than or equal to $\abs{L}$, that is,
\begin{equation}
\label{2Dineq2}
\abs{L} \geq \frac{1}{2}(\det(\va,\vx)+\det(\vx,\vb))
=\frac{1}{2}(\vx \cdot J\va + \vx \cdot (-J\vb))
=\vx \cdot \frac{1}{2}J(\va-\vb)
\end{equation}
holds for any $\vx \in K$, which means that
\begin{equation*}
\vc^\circ:=
\frac{1}{2\abs{L}}J(\va-\vb) \in K^\circ.
\end{equation*}
Similarly, we obtain
\begin{equation*}
\vc:=
\frac{1}{2\abs{L^\circ}}J(\va^\circ-\vb^\circ) \in K.
\end{equation*}
Since $\vc^\circ \cdot \vc \leq 1$, the inequality \eqref{2Dineq} holds.

If the equality of \eqref{2Dineq} holds, then the equality in \eqref{2Dineq2} holds for $\vx=\vc$.
This means that $L=\conv\{o,\va,\vc,\vb\}$ and $\vc \in L \cap \partial K$.
Moreover, $\vc \cdot J\va \geq 0$ and $\vc \cdot (-J\vb) \geq 0$ must be hold.
The argument for $L^\circ$ is the same.
\end{proof}

We use the usual notation representing the closed and open line segments joining points $\va$ and $\vb$ in $\R^2$ or $\R^3$ as follows:
\begin{equation*}
[\va, \vb] := \{ (1-t) \va+ t \vb \mid 0 \leq t \leq 1 \}, \quad
(\va, \vb):= \{ (1-t) \va+ t \vb \mid 0 < t < 1 \}.
\end{equation*}

\begin{lemma}
\label{lem:3}
Under the assumptions of Lemma \ref{lem:8}, suppose also that $\va \cdot \va^\circ=\vb \cdot \vb^\circ=1$, $\va \cdot \vb^\circ<1$, and $\vb \cdot \va^\circ<1$.
Then the inequality \eqref{2Dineq} holds with equality if and only if either the following {\upshape (i)} or {\upshape (ii)} is satisfied$:$
\begin{enumerate}[\upshape (i)]
 \item 
$L=o*([\va, \vc] \cup [\vc, \vb])$ and $L^\circ= o*([\va^\circ, \vc^\circ] \cup [\vc^\circ, \vb^\circ])=\conv \left\{o, \va^\circ, \vb^\circ\right\}$ with $0< \angle \va^\circ o \vb^\circ < \pi$, where
\begin{equation*}
\vc = \frac{1}{2\abs{L^\circ}}J(\va^\circ-\vb^\circ)
= \frac{1}{J\va^\circ \cdot \vb^\circ}J(\va^\circ-\vb^\circ), \quad
\vc^\circ = \frac{J\va^\circ \cdot \vb^\circ}{(\va-\vb) \cdot (\va^\circ-\vb^\circ)}J(\va-\vb),
\end{equation*}
 \item 
$L=o*([\va, \vc] \cup [\vc, \vb])=\conv \left\{o, \va, \vb\right\}$ with $0< \angle \va o \vb < \pi$ and $L^\circ= o*([\va^\circ, \vc^\circ] \cup [\vc^\circ, \vb^\circ])$, where
\begin{equation*}
\vc = \frac{J\va \cdot \vb}{(\va-\vb) \cdot (\va^\circ-\vb^\circ)}J(\va^\circ-\vb^\circ), \quad
\vc^\circ = \frac{1}{2\abs{L}}J(\va-\vb)= \frac{1}{J \va \cdot \vb}J(\va-\vb).
\end{equation*}
\end{enumerate}
\end{lemma}

\begin{proof}
Suppose that the equality of \eqref{2Dineq} holds.
Note that
\begin{equation*}
\vc \cdot J\va
=\frac{1}{2\abs{L^\circ}}J(\va^\circ-\vb^\circ) \cdot J\va
=\frac{1}{2\abs{L^\circ}}(\va \cdot \va^\circ -\va \cdot \vb^\circ) >0,
\end{equation*}
which yields $\vc \neq \va$.
Similarly, $\vc \neq \vb$, $\vc^\circ \neq \va^\circ$, and $\vc^\circ \neq \vb^\circ$ hold.
By Lemma \ref{lem:8}, we have
\begin{equation*}
\vc \cdot \vc^\circ=1, \quad L=o*([\va, \vc] \cup [\vc, \vb]), \quad L^\circ=o*([\va^\circ, \vc^\circ] \cup [\vc^\circ, \vb^\circ]).
\end{equation*}

If $\vc$ is a vertex of $K$, then the polar dual of $\vc$ is an edge of $K^\circ$ which contains $\vc^\circ$.
The edge contains $[\va^\circ,\vc^\circ]$ or $[\vc^\circ,\vb^\circ]$.
We may assume that it contains $[\va^\circ,\vc^\circ]$.
In particular, $\vc \cdot \va^\circ=1$.
On the other hand, observe that
\begin{equation*}
\vc \cdot \va^\circ
=\frac{1}{2\abs{L^\circ}}(J\va^\circ-J\vb^\circ) \cdot \va^\circ
=\frac{1}{2\abs{L^\circ}}(-J\vb^\circ) \cdot \va^\circ
=\frac{1}{2\abs{L^\circ}}
\det(\va^\circ,\vb^\circ).
\end{equation*}
And a similar calculation for $\vc \cdot \vb^\circ$ yields that $\vc \cdot \va^\circ=\vc \cdot \vb^\circ$.
Consequently, $\vc \cdot \vb^\circ=\vc \cdot \va^\circ=1$, and so the polar of $\vc$ also contains $[\vc^\circ,\vb^\circ]$.
This shows that $\vc^\circ$ is a point on $(\va^\circ,\vb^\circ)$, since $\vc^\circ \neq \va^\circ$ and $\vc^\circ \neq \vb^\circ$.
In this case, $L^\circ=o*([\va^\circ, \vc^\circ] \cup [\vc^\circ, \vb^\circ])$ is a triangle with $0< \angle \va^\circ o \vb^\circ < \pi$, 
that is, $L^\circ=\conv\{o,\va^\circ,\vb^\circ\}$.
Hence we have $2\abs{L^\circ}=\va^\circ \cdot (-J\vb^\circ)=J\va^\circ \cdot \vb^\circ$ and
\begin{equation*}
\vc = \frac{1}{2\abs{L^\circ}}J(\va^\circ-\vb^\circ)
= \frac{1}{J\va^\circ \cdot \vb^\circ}J(\va^\circ-\vb^\circ),
\end{equation*}
which yields that
\begin{equation*}
2\abs{L}=\vc \cdot J(\va-\vb)=\frac{1}{J\va^\circ \cdot \vb^\circ}(\va^\circ-\vb^\circ) \cdot (\va-\vb),
\end{equation*}
since the equality in \eqref{2Dineq2} holds for $\vx=\vc$.
Thus, we obtain the expression of $\vc^\circ$ in (i).

Next, if $\vc$ is not a vertex of $K$, then we obtain $L=\conv\{o,\va,\vb\}$ with $0< \angle \va o \vb < \pi$.
Then $\vc \in (\va,\vb)$, since $\vc \neq \va$ and $\vc \neq \vb$.
The polar of the edge of $K$ containing $[\va,\vb]$ is, by definition, the vector on $\partial K^\circ$ represented by $k J(\va-\vb)$, where $k$ is a positive constant.
That is nothing but $\vc^\circ$.
Hence, $\vc^\circ$ is a vertex of $K^\circ$.
Since $\vc^\circ \neq \va^\circ$ and $\vc^\circ \neq \vb^\circ$, $L^\circ=o*([\va^\circ, \vc^\circ] \cup [\vc^\circ, \vb^\circ])$ must be a quadrangle.
The expressions of $\vc^\circ$ and $\vc$ in (ii) are obtained by the same calculation as the previous case.
\end{proof}

Lemma \ref{lem:3} enables us to extend \cite{IS2}*{Proposition 3.6} to the case that $\va \parallel \va^\circ$ and $\vb \parallel \vb^\circ$ do not hold necessarily.
Let $H \subset \R^3$ be an oriented plane with $o \in H$ and $\Pi_H$ denotes the orthogonal projection from $\R^3$ onto $H$.
We take two vectors $\va, \vb \in \R^3 \setminus \{o\}$ such that $\Pi_H \va$, $\Pi_H \vb \in H$ are not zero and not positively proportional.
For any $K \in \mathcal{K}$ with $o \in \mathrm{int}(K)$, we condider the boundary $\partial \Pi_H(K)$ of the convex body $\Pi_H(K) \subset H$, oriented counterclockwise. 
Let $\hat{\va}$ (resp.\ $\hat{\vb}$) be the intersection of $\partial \Pi_H(K)$ and the half line $\R_+(\Pi_H \va)$ (resp.\ $\R_+(\Pi_H \vb)$).
Then, the symbol
\begin{equation*}
\mathcal{C}_K(H; \va,\vb)
\end{equation*}
denotes the locus which moves from $\hat{\va}$ to $\hat{\vb}$ with the same orientation of $\partial \Pi_H(K)$.

\begin{remark}
\label{rem:C}
In what follows, whenever $\va \nparallel \vb$, we equip the orientation with $H:=\Span \{\va, \vb \} \subset \R^3$ such that $\va, \vb, \va \times \vb$ form the right-hand system.
Moreover, if $K \in \mathcal{K}$ with $o \in \mathrm{int}(K)$ satisfies that $K \cap H=\Pi_H(K)$, then we have
\begin{equation*}
\mathcal{C}_K(H; \va,\vb) = \mathcal{C}_K(\va,\vb).
\end{equation*}
In general, this equality does not hold.
\end{remark}

\begin{lemma}
\label{lem:4}
Let $K \in \mathcal{K}$ and $\va, \vb \in \partial K$ with $\va \nparallel \vb$.
Assume that $\va^\circ, \vb^\circ \in \partial K^\circ$ satisfies $\va \cdot \va^\circ=\vb \cdot \vb^\circ=1$, $\va \cdot \vb^\circ<1$, and $\vb \cdot \va^\circ<1$.
Let $H \subset \R^3$ be the plane spanned by $\va$ and $\vb$, and $\Pi_H$ denotes the orthogonal projection onto $H$.
Moreover, $\vn$ denotes the unit normal vector of $H$ defined by $\va \times \vb/\norm{\va \times \vb}$.
Then setting
\begin{equation*}
L:=o*\mathcal{C}_K(\va,\vb),\quad 
L^\circ:=o*\mathcal{C}_{K^\circ}(H; \va^\circ, \vb^\circ) \subset H,
\end{equation*}
we have
\begin{equation*}
\abs{L} \, \abs{L^\circ} \geq \frac{1}{4}(2-\va \cdot \vb^\circ-\va^\circ \cdot \vb)
=\frac{1}{4}(2-\va \cdot \Pi_H\vb^\circ-\Pi_H\va^\circ \cdot \vb).
\end{equation*}
The equality holds if and only if either the following {\upshape (i)} or {\upshape (ii)} is satisfied$:$
\begin{enumerate}[\upshape (i)]
 \item
$L=o*([\va, \vc] \cup [\vc, \vb])$ and 
$L^\circ 
=\conv \left\{o, \Pi_H\va^\circ, \Pi_H\vb^\circ\right\}$ with $0< \angle (\Pi_H\va^\circ) o (\Pi_H\vb^\circ) < \pi$, where
\begin{equation*}
\begin{aligned}
\vc &= \frac{1}{(\vn \times \Pi_H\va^\circ) \cdot \Pi_H\vb^\circ}\vn \times \Pi_H(\va^\circ-\vb^\circ)
= \frac{1}{\det(\vn,\va^\circ,\vb^\circ)}\vn \times (\va^\circ-\vb^\circ), \\
\vc^\circ &= \frac{(\vn \times \Pi_H\va^\circ) \cdot \Pi_H\vb^\circ}{(\va-\vb) \cdot \Pi_H(\va^\circ-\vb^\circ)}\vn \times (\va-\vb)
= \frac{\det(\vn,\va^\circ,\vb^\circ)}{(\va-\vb) \cdot (\va^\circ-\vb^\circ)}\vn \times (\va-\vb).
\end{aligned}
\end{equation*}
 \item
$L=\conv \left\{o, \va, \vb\right\}$ and $L^\circ= o*([\Pi_H \va^\circ, \vc^\circ] \cup [\vc^\circ, \Pi_H \vb^\circ])$, where
\begin{equation*}
\begin{aligned}
\vc &= \frac{(\vn \times \va) \cdot \vb}{(\va-\vb) \cdot \Pi_H(\va^\circ-\vb^\circ)}\vn \times \Pi_H(\va^\circ-\vb^\circ)
= \frac{\det(\vn,\va,\vb)}{(\va-\vb) \cdot (\va^\circ-\vb^\circ)}\vn \times (\va^\circ-\vb^\circ), \\
\vc^\circ &= \frac{1}{\det(\vn,\va,\vb)}\vn \times (\va-\vb).
\end{aligned}
\end{equation*}
\end{enumerate}
\end{lemma}

\begin{proof}
By the assumptions, we obtain
\begin{equation*}
\va \cdot \Pi_H\va^\circ=\vb \cdot \Pi_H\vb^\circ=1,\quad \va \cdot \Pi_H\vb^\circ<1,\quad \vb \cdot \Pi_H\va^\circ<1,
\end{equation*}
which implies that $\Pi_H\va^\circ$, $\Pi_H\vb^\circ \in H$ are not positively proportional.
The polar of the convex body $K \cap H$ in $H \cong \R^2$ is $\Pi_H(K^\circ) \subset H$.
Note that $J \vx = \vn \times \vx$ for any vector $\vx \in H \subset \R^3$
 and $(\vn \times \vx)\cdot \vy = \det(\vn,\vx,\vy)$ for any vectors $\vx, \vy \in \R^3$.
Hence we can apply Lemma \ref{lem:3} to the pair $K \cap H$ and $\Pi_H(K^\circ)$ on $H$, and the results hold.
\end{proof}

\begin{remark}
Under the assumption of Lemma \ref{lem:4}, it always holds that $0< \angle \va o \vb < \pi$.
However, the situation that $\angle \va^\circ o \vb^\circ \geq \pi$ can happen.
\end{remark}

\begin{lemma}
\label{lem:7}
Let $K \in \mathcal{K}$ and we take points $\va_1,\ldots,\va_m,\va^\circ_1,\ldots,\va^\circ_\ell \in \R^3 \setminus \{o\}$.
Assume that $\mathcal{C}:=\mathcal{C}_K(\va_1,\ldots,\va_m,\va_1)$ is a simple closed curve on $\partial K$
such that
\begin{equation*}
\mathcal{C} = \bigcup_{i=1}^m [\va_i, \va_{i+1}] \subset \partial K
\end{equation*}
with $\va_{m+1}:=\va_1$,
and also $\mathcal{C}^\circ:=\mathcal{C}_{K^\circ}(\va^\circ_1,\ldots,\va^\circ_\ell,\va^\circ_1)$ on $\partial K^\circ$ is.
Then we have
\begin{equation}
\label{3Dineq}
\abs{o*\mathcal{S}_K(\mathcal{C})} \, \abs{o*\mathcal{S}_{K^\circ}(\mathcal{C}^\circ)}
\geq \frac{1}{9} \, \overline{\mathcal{C}} \cdot \overline{\mathcal{C}^\circ}.
\end{equation}
Moreover, if the equality of \eqref{3Dineq} holds, then there exist positive constants $\tau$ and $\tau^\circ$ such that
\begin{equation*}
\begin{aligned}
\mathcal{S}_K(\mathcal{C}) &= (\tau \overline{\mathcal{C}^\circ})* \mathcal{C}, &
\det({\overline{\mathcal{C}^\circ}, \va_i, \va_{i+1}}) &\geq 0 \quad (i=1,\dots, m), \\
\mathcal{S}_{K^\circ}(\mathcal{C^\circ}) &= (\tau^\circ \overline{\mathcal{C}})* \mathcal{C}^\circ, &
\det({\overline{\mathcal{C}}, \va^\circ_j, \va^\circ_{j+1}}) &\geq 0 \quad (j=1,\dots, \ell). 
\end{aligned}
\end{equation*}
\end{lemma}
\begin{proof}
By definition, $\overline{\mathcal{C}}=(1/2)\sum_{i=1}^m \va_i \times \va_{i+1}$.
For any point $\vx \in K$ we have
\begin{equation}
\label{3Dineq2}
\frac{1}{3} \vx \cdot \overline{\mathcal{C}}
= \frac{1}{6}\sum_{i=1}^m \vx \cdot (\va_i \times \va_{i+1})
= \frac{1}{6}\sum_{i=1}^m \det({\vx, \va_i, \va_{i+1}})
\leq \abs{o*\mathcal{S}_K(\mathcal{C})},
\end{equation}
which means that
\begin{equation*}
\vc^\circ:=
\frac{1}{3\abs{o*\mathcal{S}_K(\mathcal{C})}} \, \overline{\mathcal{C}} \in K^\circ.
\end{equation*}
Similarly, we obtain
\begin{equation*}
\vc:=
\frac{1}{3\abs{o*\mathcal{S}_{K^\circ}(\mathcal{C}^\circ)}} \, \overline{\mathcal{C}^\circ} \in K.
\end{equation*}
Since $\vc^\circ \cdot \vc \leq 1$, the inequality \eqref{3Dineq} holds.

If the equality of \eqref{3Dineq} holds, then the equality in \eqref{3Dineq2} holds for $\vx=\vc$,
which implies that $\vc \in \partial K$, $o*\mathcal{S}_K(\mathcal{C})=(o*\mathcal{C})* \vc=o*(\vc * \mathcal{C})$ and $\det({\vc, \va_i, \va_{i+1}}) \geq 0$ for all $i$.
Hence, $\mathcal{S}_K(\mathcal{C})=\vc * \mathcal{C}$.
The argument for $o*\mathcal{S}_{K^\circ}(\mathcal{C^\circ})$ is the same.
\end{proof}

\section{The case with $D_2$-symmetry}
\label{sec:3}

\subsection{Settings and preliminary calculations}

Without loss of generality, we may assume that the subgroup $D_2 \subset SO(3)$ consists of the following four elements:
\begin{equation*}
E:=
\begin{pmatrix}
1 & 0 & 0 \\
0 & 1 & 0 \\
0 & 0 & 1
\end{pmatrix}, \quad
R_{23}:=
\begin{pmatrix}
1 & 0 & 0 \\
0 & -1 & 0 \\
0 & 0 & -1
\end{pmatrix}, \quad
R_{13}:=
\begin{pmatrix}
-1 & 0 & 0 \\
0 & 1 & 0 \\
0 & 0 & -1
\end{pmatrix}, \quad
R_{12}:=
\begin{pmatrix}
-1 & 0 & 0 \\
0 & -1 & 0 \\
0 & 0 & 1
\end{pmatrix}
.
\end{equation*}
Let $K \in \checkK^3(D_2)$. 
For positive real numbers $a$, $b$, and $c$, we put 
\begin{equation*}
\vv_{---}:=
\begin{pmatrix}
-a \\ -b \\ -c
\end{pmatrix}, \quad
\vv_{-++}:=
\begin{pmatrix}
-a \\ +b \\ +c
\end{pmatrix}, \quad
\vv_{+-+}:=
\begin{pmatrix}
+a \\ -b \\ +c
\end{pmatrix}, \quad
\vv_{++-}:=
\begin{pmatrix}
+a \\ +b \\ -c
\end{pmatrix}.
\end{equation*}
These are vertices of a $3$-simplex satisfying
\begin{equation}
\label{eq:2}
\vv_{-++} = R_{23} \vv_{---}, \quad
\vv_{+-+} = R_{13} \vv_{---}, \quad
\vv_{++-} = R_{12} \vv_{---}.
\end{equation}

We can divide $K$ into the four parts
\begin{equation*}
\begin{aligned}
K_0 &:= K \cap \pos(\vv_{-++}, \vv_{+-+}, \vv_{++-}), &
R_{23} K_0 &= K \cap \pos(\vv_{---}, \vv_{++-}, \vv_{+-+}), \\
R_{13} K_0 &= K \cap \pos(\vv_{++-}, \vv_{---}, \vv_{-++}), &
R_{12} K_0 &= K \cap \pos(\vv_{+-+}, \vv_{-++}, \vv_{---}).
\end{aligned}
\end{equation*}
Since they are congruent, by the volume estimate (Lemma \ref{lem:1}), we have
\begin{equation}
\label{eq:39}
\begin{aligned}
& \abs{K} \, \abs{K^\circ} = 16 \abs{K_0} \, \abs{\Lambda(K_0)} \\
&\geq
\frac{16}{9}
\left(
\overline{\mathcal{C}(\vv_{+-+}, \vv_{++-})}
+
\overline{\mathcal{C}(\vv_{++-},\vv_{-++})}
+
\overline{\mathcal{C}(\vv_{-++}, \vv_{+-+})}
\right) \\
&\quad \cdot
\left(
\overline{\Lambda (\mathcal{C}(\vv_{+-+}, \vv_{++-}))}
+
\overline{\Lambda (\mathcal{C}(\vv_{++-},\vv_{-++}))}
+
\overline{\Lambda (\mathcal{C}(\vv_{-++}, \vv_{+-+}))}
\right).
\end{aligned}
\end{equation}
We note that
\begin{equation*}
\overline{\mathcal{C}(\vv_{+-+}, \vv_{++-})}
\parallel
\vv_{+-+} \times \vv_{++-}, \quad
\vv_{+-+} \times \vv_{++-}
=
\begin{pmatrix}
0 \\ 2ca \\ 2ab
\end{pmatrix}
\end{equation*}
and obtain
\begin{equation*}
\begin{aligned}
\overline{\mathcal{C}(\vv_{+-+}, \vv_{++-})}
&= \frac{|o*\mathcal{C}(\vv_{+-+}, \vv_{++-})|}{\sqrt{c^2a^2+a^2b^2}}
\begin{pmatrix}
0 \\ ca \\ ab
\end{pmatrix}, \quad
\overline{\mathcal{C}(\vv_{++-},\vv_{-++})}
= \frac{|o*\mathcal{C}(\vv_{++-},\vv_{-++})|}{\sqrt{b^2c^2+a^2b^2}}
\begin{pmatrix}
bc \\ 0 \\ ab
\end{pmatrix}, \\
\overline{\mathcal{C}(\vv_{-++}, \vv_{+-+})}
&= \frac{|o*\mathcal{C}(\vv_{-++}, \vv_{+-+})|}{\sqrt{b^2c^2+c^2a^2}}
\begin{pmatrix}
bc \\ ca \\ 0
\end{pmatrix}
\end{aligned}
\end{equation*}
(see \cite[(5)]{IS2}).
Now we choose $a$, $b$, and $c$ which satisfy that 
\begin{equation}
\label{eq:1}
S_K:=
\frac{|o*\mathcal{C}(\vv_{+-+}, \vv_{++-})|}{\sqrt{c^2a^2+a^2b^2}}=
\frac{|o*\mathcal{C}(\vv_{++-},\vv_{-++})|}{\sqrt{b^2c^2+a^2b^2}}=
\frac{|o*\mathcal{C}(\vv_{-++}, \vv_{+-+})|}{\sqrt{b^2c^2+c^2a^2}}
\end{equation}
and show the inequality of Theorem \ref{thm:1}.
Indeed, the following proposition holds.
\begin{proposition}
\label{prop:1}
For any $K \in \checkK^3(D_2)$ there exist positive constants $a$, $b$, and $c$ satisfying the condition \eqref{eq:1}.
\end{proposition}
We give the proof in Section \ref{sec:3.3}.
Owing to the equipartitional condition \eqref{eq:1}, we can transform the right-hand side of \eqref{eq:39} into a computable form as in \eqref{eq:37} below.
There, the cross terms of \eqref{eq:39} do not appear.

\subsection{Proof of the inequality}

\begin{proposition}
\label{prop:2}
For any $K \in \checkK^3(D_2)$, the inequality in Theorem \ref{thm:1} holds.
\end{proposition}

\begin{proof}
Let $K \in \checkK^3(D_2)$.
By Proposition \ref{prop:1}, there exist
$a,b,c>0$ satisfying the condition \eqref{eq:1}.
Then, it holds that
\begin{equation}
\label{eq:35}
\overline{\mathcal{C}(\vv_{+-+}, \vv_{++-})}
+
\overline{\mathcal{C}(\vv_{++-},\vv_{-++})}
+
\overline{\mathcal{C}(\vv_{-++}, \vv_{+-+})}
=
2 S_K
\begin{pmatrix}
bc \\ ca \\ ab
\end{pmatrix}.
\end{equation}
Putting
\begin{equation*}
\begin{pmatrix}
\alpha^\circ \\ \beta^\circ \\ \gamma^\circ
\end{pmatrix}
:=
\overline{\Lambda (\mathcal{C}(\vv_{+-+}, \vv_{++-}))},
\end{equation*}
since
\begin{equation*}
R_{23} \overline{\Lambda (\mathcal{C}(\vv_{+-+}, \vv_{++-}))}
=
\overline{\Lambda (\mathcal{C}(\vv_{++-}, \vv_{+-+}))}
=
-
\overline{\Lambda (\mathcal{C}(\vv_{+-+}, \vv_{++-}))},
\end{equation*}
we have $\alpha^\circ=0$. Hence we get
\begin{equation*}
2 S_K
\begin{pmatrix}
bc \\ ca \\ ab
\end{pmatrix}
\cdot
\overline{\Lambda (\mathcal{C}(\vv_{+-+}, \vv_{++-}))}
=
2 S_K(ca \beta^\circ + ab \gamma^\circ)
=
2 \, \overline{\mathcal{C}(\vv_{+-+}, \vv_{++-})}
\cdot
\overline{\Lambda (\mathcal{C}(\vv_{+-+}, \vv_{++-}))}.
\end{equation*}
Combining it with \eqref{eq:35}, we have
\begin{equation*}
\begin{aligned}
&\left(
\overline{\mathcal{C}(\vv_{+-+}, \vv_{++-})}
+
\overline{\mathcal{C}(\vv_{++-},\vv_{-++})}
+
\overline{\mathcal{C}(\vv_{-++}, \vv_{+-+})}
\right)\cdot
\overline{\Lambda (\mathcal{C}(\vv_{+-+}, \vv_{++-}))} \\
& =
2 \, \overline{\mathcal{C}(\vv_{+-+}, \vv_{++-})}
\cdot
\overline{\Lambda (\mathcal{C}(\vv_{+-+}, \vv_{++-}))}.
\end{aligned}
\end{equation*}
Similarly, we obtain
\begin{equation}
\label{eq:37}
\begin{aligned}
&\left(
\overline{\mathcal{C}(\vv_{+-+}, \vv_{++-})}
+
\overline{\mathcal{C}(\vv_{++-},\vv_{-++})}
+
\overline{\mathcal{C}(\vv_{-++}, \vv_{+-+})}
\right) \\
&\cdot
\left(
\overline{\Lambda (\mathcal{C}(\vv_{+-+}, \vv_{++-}))}
+
\overline{\Lambda (\mathcal{C}(\vv_{++-},\vv_{-++}))}
+
\overline{\Lambda (\mathcal{C}(\vv_{-++}, \vv_{+-+}))}
\right) \\
& =
2 \, \overline{\mathcal{C}(\vv_{+-+}, \vv_{++-})}
\cdot
\overline{\Lambda (\mathcal{C}(\vv_{+-+}, \vv_{++-}))}
+
2 \, \overline{\mathcal{C}(\vv_{++-}, \vv_{-++})}
\cdot
\overline{\Lambda (\mathcal{C}(\vv_{++-}, \vv_{-++}))} \\
&+
2 \, \overline{\mathcal{C}(\vv_{-++}, \vv_{+-+})}
\cdot
\overline{\Lambda (\mathcal{C}(\vv_{-++}, \vv_{+-+}))}.
\end{aligned}
\end{equation}
By a dilation of $K$, we may assume that $\vv_{---} \in \partial K$.
Then 
$\vv_{+-+}, \vv_{++-} \in \partial K$ from \eqref{eq:2}, 
and by Lemma \ref{lem:2}, it holds that
\begin{equation*}
2 \, \overline{\mathcal{C}(\vv_{+-+}, \vv_{++-})}
\cdot
\overline{\Lambda (\mathcal{C}(\vv_{+-+}, \vv_{++-}))}
\geq \frac{1}{2} (\vv_{+-+}-\vv_{++-})\cdot (\Lambda \vv_{+-+}- \Lambda \vv_{++-}).
\end{equation*}
Putting
\begin{equation*}
\begin{pmatrix}
a^\circ \\ b^\circ \\ c^\circ
\end{pmatrix}
:= -\Lambda \vv_{---},
\end{equation*}
we have 
\begin{equation*}
\vv_{---} \cdot \Lambda \vv_{---} = a a^\circ+b b^\circ+c c^\circ=1,
\end{equation*}
and 
\begin{equation*}
\begin{aligned}
\frac{1}{2} (\vv_{+-+}-\vv_{++-})\cdot (\Lambda \vv_{+-+}- \Lambda \vv_{++-})
&=
\begin{pmatrix}
0 \\ -b \\ c
\end{pmatrix}
\cdot
(\Lambda R_{13} \vv_{---} - \Lambda R_{12} \vv_{---}) \\
&=
\begin{pmatrix}
0 \\ -b \\ c
\end{pmatrix}
\cdot
\begin{pmatrix}
0 \\ -2 b^\circ \\ 2c^\circ
\end{pmatrix}
=
2 b b^\circ + 2 c c^\circ,
\end{aligned}
\end{equation*}
where the second equality holds due to Lemma \ref{lem:6}.
Hence, 
\begin{equation}
\label{eq:38}
2 \, \overline{\mathcal{C}(\vv_{+-+}, \vv_{++-})}
\cdot
\overline{\Lambda (\mathcal{C}(\vv_{+-+}, \vv_{++-}))}
\geq
2 b b^\circ + 2 c c^\circ.
\end{equation}
Similarly, we obtain
\begin{equation}
\label{eq:36}
\begin{aligned}
&2 \, \overline{\mathcal{C}(\vv_{++-}, \vv_{-++})}
\cdot
\overline{\Lambda (\mathcal{C}(\vv_{++-}, \vv_{-++}))}
\geq
2a a^\circ + 2 c c^\circ, \\
&2 \, \overline{\mathcal{C}(\vv_{-++}, \vv_{+-+})}
\cdot
\overline{\Lambda (\mathcal{C}(\vv_{-++}, \vv_{+-+}))} 
\geq
2 a a^\circ + 2 b b^\circ.
\end{aligned}
\end{equation}
It follows from \eqref{eq:39}, \eqref{eq:37}, \eqref{eq:38}, and \eqref{eq:36} that 
\begin{equation*}
\abs{K} \, \abs{K^\circ} \geq \frac{16}{9} 
\left(
4 a a^\circ
+
4 b b^\circ
+
4 c c^\circ
\right)
= \frac{64}{9},
\end{equation*}
which completes the proof.
\end{proof}

\subsection{Equipartition}
\label{sec:3.3}

Now we shall prove Proposition \ref{prop:1}.
In this subsection, $\R^3_{+++}$ denotes the first octant of $\R^3$ and $B \subset \R^3$ denotes the unit closed ball centered at the origin.

\begin{proof}[Proof of Proposition \ref{prop:1}]
For $\vp:=(a,b,c)^{\mathrm{T}} \in \R^3_{+++}$, we 
consider the continuous functions
\begin{equation*}
S_1(\vp):=
\frac{|o*\mathcal{C}(\vv_{+-+}, \vv_{++-})|}{\sqrt{c^2a^2+a^2b^2}}, \quad
S_2(\vp):=
\frac{|o*\mathcal{C}(\vv_{++-},\vv_{-++})|}{\sqrt{b^2c^2+a^2b^2}}, \quad
S_3(\vp):=
\frac{|o*\mathcal{C}(\vv_{-++}, \vv_{+-+})|}{\sqrt{b^2c^2+c^2a^2}}
\end{equation*}
and take $r>0$ such that $r B \subset K \subset r^{-1} B$.
Then, we see that
\begin{equation*}
|o*\mathcal{C}_{rB}(\vv_{+-+}, \vv_{++-})|
\leq
|o*\mathcal{C}_K(\vv_{+-+}, \vv_{++-})|
\leq
|o*\mathcal{C}_{r^{-1}B}(\vv_{+-+}, \vv_{++-})|.
\end{equation*}
Since $B$ is the unit ball, we have
\begin{equation*}
|o*\mathcal{C}_{B}(\vv_{+-+}, \vv_{++-})|
=
\frac{1}{2}\arccos\left(
\frac{\vv_{+-+} \cdot \vv_{++-}}{\norm{\vv_{+-+}} \, \norm{\vv_{++-}}}
\right)
=
\frac{1}{2}
\arccos \left(
\frac{a^2-b^2-c^2}{a^2+b^2+c^2}
\right).
\end{equation*}
Thus, putting 
\begin{equation*}
f(\vp)
:=\frac{1}{2\sqrt{c^2 a^2+a^2 b^2}}
\arccos \left(
\frac{a^2-b^2-c^2}{a^2+b^2+c^2}
\right),
\end{equation*}
we get 
\begin{equation*}
r^2 f(\vp)
\leq S_1(\vp)
\leq
r^{-2} f(\vp).
\end{equation*}
Similarly, by using the notation
\begin{equation*}
R:=
\begin{pmatrix}
0 & 0 & 1 \\
1 & 0 & 0 \\
0 & 1 & 0 \\
\end{pmatrix}, \quad
R \vp=
\begin{pmatrix}
c \\ a \\ b
\end{pmatrix}, \quad
R^2 \vp=
\begin{pmatrix}
b \\ c \\ a
\end{pmatrix},
\end{equation*}
we also obtain
\begin{equation*}
r^2 f(R^2\vp)
\leq S_2(\vp)
\leq
r^{-2} f(R^2\vp), \quad
r^2 f(R\vp)
\leq S_3(\vp)
\leq
r^{-2} f(R\vp).
\end{equation*}
To get the required condition \eqref{eq:1}, we introduce a continuous map on $\R^3_{+++}$ defined by
\begin{equation*}
F(\vp):=
S_1(\vp)
\begin{pmatrix}
1 \\ 0
\end{pmatrix}
+
S_2(\vp)
\begin{pmatrix}
-1/2 \\ \sqrt{3}/2
\end{pmatrix}
+
S_3(\vp)
\begin{pmatrix}
-1/2 \\ -\sqrt{3}/2
\end{pmatrix}.
\end{equation*}
If the map $F: \R^3_{+++} \to \R^2$ has a zero point $\vp_0$, then \eqref{eq:1} holds at $\vp_0$.
We choose sufficiently small $\epsilon>0$ such that
\begin{equation}
\label{eq:40}
\frac{r^2 \pi}{8\sqrt{2}}
- r^{-2} \max\{\pi \sqrt{\epsilon}, 2 \arccos(1-20\epsilon)\} >0, \quad
\epsilon < \frac{1}{4},
\end{equation}
then $0 < \sqrt{\epsilon} <1/2$. 
Note that this constant $\epsilon>0$ depends only on $r$.
We set
\begin{equation*}
\va:=
\begin{pmatrix}
1 \\ \epsilon \\ \epsilon
\end{pmatrix}, \quad
\vb:=
\begin{pmatrix}
\epsilon \\ 1 \\ \epsilon
\end{pmatrix}, \quad
\vc:=
\begin{pmatrix}
\epsilon \\ \epsilon \\ 1
\end{pmatrix}
\in \R^3_{+++},
\end{equation*}
then $R \va = \vb$, $R \vb = \vc$, and $R \vc= \va$.
Let us consider the triangle $T:=\conv\{\va, \vb, \vc\} \subset \R^3_{+++}$ and calculate the winding number of $F(\vp)$ on
$\partial T = [\va, \vb] \cup [\vb, \vc] \cup [\vc, \va]$, which is homeomorphic to $S^1$.
If the winding number is not zero, then we get a zero point in $T$.
Indeed, 
by \cite{OR}*{Proposition 4.4}, 
for the continuous map $F: T \rightarrow \R^2$, if $o \notin F(\partial T)$ and 
the winding number $W(F|\partial T,o)$ of $F(\partial T)$ around the origin $o \in \R^2$ is not $0$, then we have
$o \in F(T)$. Hence, $F$ has a zero point in the interior of $T$.

Now we are going to show that $o \notin F(\partial T)$ and $W(F|\partial T,o) \neq 0$.
We put 
\begin{equation*}
\va':= \frac{\vb + \vc}{2}, \quad
\vb':= \frac{\vc + \va}{2}, \quad
\vc':= \frac{\va + \vb}{2}.
\end{equation*}
These are the midpoints of the line segments $[\vb, \vc]$, $[\vc, \va]$, $[\va, \vb]$, respectively.
For $\vp=(a,b,c)^\mathrm{T}:=(1-t) \va + t \vb$ $(0 \leq t \leq 1/2)$ on the segment $[\va, \vc']$, we have
\begin{equation*}
\begin{aligned}
F(\vp) \cdot 
\begin{pmatrix}
-1 \\ 0
\end{pmatrix}
&=\frac{1}{2} \left(
S_2(\vp)
+
S_3(\vp)
\right)
-
S_1(\vp)
\geq
\frac{1}{2}S_3(\vp)
-
S_1(\vp)
\geq 
\frac{r^2}{2} f(R \vp)
-r^{-2} f(\vp) \\
&=
\frac{r^2}{2} f((1-t)\vb + t \vc)
-r^{-2} f((1-t)\va + t \vb).
\end{aligned}
\end{equation*}
If $\sqrt{\epsilon} \leq t \leq 1/2$, then
\begin{equation*}
\sqrt{c^2 a^2 + a^2 b^2} 
\geq a b
=(1-t+\epsilon t)((1-t)\epsilon + t)
 \geq (1-t)t \geq \frac{\sqrt{\epsilon}}{2}
\end{equation*}
holds and we see that
\begin{equation*}
f((1-t)\va + t \vb) = f(\vp) \leq \frac{\pi}{\sqrt{\epsilon}}
\quad \left(\sqrt{\epsilon} \leq t \leq \frac{1}{2}\right).
\end{equation*}
If $0 \leq t \leq \sqrt{\epsilon} \,(<1/2)$, then we have
\begin{equation*}
\sqrt{c^2 a^2 + a^2 b^2} 
\geq a b
=(1-t+\epsilon t)((1-t)\epsilon + t)
\geq (1-t)^2\epsilon
\geq \frac{\epsilon}{4}
\end{equation*}
and
\begin{equation*}
\begin{aligned}
&\frac{a^2-b^2-c^2}{a^2+b^2+c^2}
=1- \frac{2(b^2+c^2)}{a^2+b^2+c^2}
\geq
1- \frac{2(b^2+c^2)}{a^2}
=
1- \frac{2(((1-t)\epsilon + t)^2+\epsilon^2)}{(1-t+\epsilon t)^2} \\
& \geq
1- \frac{2((\epsilon + \sqrt{\epsilon})^2+\epsilon^2)}{(1-t)^2}
\geq 1- 8(2 \epsilon^2 + 2 \epsilon \sqrt{\epsilon} + \epsilon)
\geq 1-20\epsilon,
\end{aligned}
\end{equation*}
which yield that
\begin{equation*}
f((1-t)\va + t \vb) \leq \frac{2}{\epsilon} \arccos(1-20\epsilon)
\quad
(0 \leq t \leq \sqrt{\epsilon}).
\end{equation*}
Thus we obtain
\begin{equation}
\label{eq:5}
f((1-t)\va + t \vb) \leq 
\max\left\{
\frac{\pi}{\sqrt{\epsilon}},
\frac{2}{\epsilon} \arccos(1-20\epsilon)
\right\}
\quad
\left(0 \leq t \leq \frac{1}{2}\right).
\end{equation}
On the other hand, since
$0 \leq t \leq 1/2$ and $\epsilon^2 < 1/4 \leq (1-t)^2 \leq 1$, 
we have
\begin{equation*}
(1-t+\epsilon t)^2 + (\epsilon - t \epsilon + t)^2
\leq 2, \quad
\epsilon^2 - (1-t +\epsilon t)^2 -(\epsilon - t \epsilon + t)^2
\leq \epsilon ^2 - (1-t)^2 \leq 0.
\end{equation*}
Hence, it holds that
\begin{equation*}
\begin{aligned}
&f((1-t)\vb + t \vc) \\
&=
\frac{1}{2 \epsilon \sqrt{(1-t+\epsilon t)^2 + (\epsilon - t \epsilon + t)^2}} \arccos
\left(
\frac{\epsilon^2 - (1-t +\epsilon t)^2 -(\epsilon - t \epsilon + t)^2}{\epsilon^2 + (1-t +\epsilon t)^2 +(\epsilon - t \epsilon + t)^2}
\right) 
\geq 
\frac{\pi}{4\sqrt{2} \epsilon}.
\end{aligned}
\end{equation*}
Combining it with \eqref{eq:5} and \eqref{eq:40}, we obtain
\begin{equation}
\label{eq:6}
\begin{aligned}
&\frac{r^2}{2} f((1-t)\vb + t \vc)
-r^{-2} f((1-t)\va + t \vb) \\
&\geq
\frac{1}{\epsilon}
\left(
\frac{r^2 \pi}{8\sqrt{2}}
- r^{-2} \max\{\pi \sqrt{\epsilon}, 2 \arccos(1-20\epsilon)\}
\right) >0 \quad \left(0 \leq t \leq \frac{1}{2}\right).
\end{aligned}
\end{equation}
Since $f((a,b,c)^\mathrm{T})=f((a,c,b)^\mathrm{T})$ by definition, 
for $\lambda_i \geq 0$ $(i=1,2,3)$ such that $(\lambda_1, \lambda_2, \lambda_3)\not=(0,0,0)$, 
we have
\begin{equation}
\label{eq:14}
f(\lambda_1 \va + \lambda_2 \vb + \lambda_3 \vc)
=
f(\lambda_1 \va + \lambda_3 \vb + \lambda_2 \vc).
\end{equation}
For $\vp=(1-t)\va + t \vc$ $(0 \leq t \leq 1/2)$ on $[\va, \vb']$, 
by \eqref{eq:6} and \eqref{eq:14}, 
we have
\begin{equation*}
\begin{aligned}
F(\vp) \cdot 
\begin{pmatrix}
-1 \\ 0
\end{pmatrix}
&=\frac{1}{2} \left(
S_2(\vp)
+
S_3(\vp)
\right)
-
S_1(\vp)
\geq
\frac{1}{2}S_2(\vp)
-
S_1(\vp)
\geq 
\frac{r^2}{2} f(R^2 \vp)
-r^{-2} f(\vp) \\
&=
\frac{r^2}{2} f((1-t)\vc + t \vb)
-r^{-2} f((1-t)\va + t \vc) \\
&=
\frac{r^2}{2} f((1-t)\vb + t \vc)
-r^{-2} f((1-t)\va + t \vb)>0
\quad \left(0 \leq t \leq \frac{1}{2}\right).
\end{aligned}
\end{equation*}
Next, for $\vp=(1-t)\vb+t \vc$ $(0 \leq t \leq 1/2)$ on $[\vb, \va']$, we have
\begin{equation*}
\begin{aligned}
F(\vp) \cdot 
\begin{pmatrix}
1/2 \\ -\sqrt{3}/2
\end{pmatrix}
&=\frac{1}{2} \left(
S_3(\vp)
+
S_1(\vp)
\right)
-
S_2(\vp)
\geq
\frac{1}{2}S_1(\vp)
-
S_2(\vp)
\geq 
\frac{r^2}{2} f(\vp)
-r^{-2} f(R^2 \vp) \\
&=
\frac{r^2}{2} f((1-t)\vb + t \vc)
-r^{-2} f((1-t)\va + t \vb) >0
\quad \left(0 \leq t \leq \frac{1}{2}\right).
\end{aligned}
\end{equation*}
As for $\vp=(1-t)\vb+t \va$ $(0 \leq t \leq 1/2)$ on $[\vb, \vc']$, we have
\begin{equation*}
\begin{aligned}
F(\vp) \cdot 
\begin{pmatrix}
1/2 \\ -\sqrt{3}/2
\end{pmatrix}
&=\frac{1}{2} \left(
S_3(\vp)
+
S_1(\vp)
\right)
-
S_2(\vp)
\geq
\frac{1}{2}S_3(\vp)
-
S_2(\vp)
\geq 
\frac{r^2}{2} f(R\vp)
-r^{-2} f(R^2 \vp) \\
&=
\frac{r^2}{2} f((1-t)\vc + t \vb)
-r^{-2} f((1-t)\va + t \vc) \\
&=
\frac{r^2}{2} f((1-t)\vb + t \vc)
-r^{-2} f((1-t)\va + t \vb) 
>0
\quad \left(0 \leq t \leq \frac{1}{2}\right).
\end{aligned}
\end{equation*}
Similarly, for any point $\vp$ on $[\vc, \vb'] \cup [\vc, \va']$, we obtain
\begin{equation*}
F(\vp) \cdot 
\begin{pmatrix}
1/2 \\ \sqrt{3}/2
\end{pmatrix}
>0.
\end{equation*}
Now we divide the boundary $\partial T$ of the triangle $T$ into 
$[\va, \vb'] \cup [\va, \vc']$,
$[\vb, \vc'] \cup [\vb, \va']$, and
$[\vc, \va'] \cup [\vc, \vb']$. Then, from the above results, we obtain
\begin{equation*}
\begin{aligned}
&F(\vp) \cdot
\begin{pmatrix}
-1 \\ 0
\end{pmatrix}
>0 \text{ for } \vp \in [\va, \vb'] \cup [\va, \vc'], \quad
F(\vp) \cdot
\begin{pmatrix}
1/2 \\ -\sqrt{3}/2
\end{pmatrix}
>0 \text{ for } \vp \in [\vb, \vc'] \cup [\vb, \va'], \\
&F(\vp) \cdot
\begin{pmatrix}
1/2 \\ \sqrt{3}/2
\end{pmatrix}
>0 \text{ for } \vp \in [\vc, \va'] \cup [\vc, \vb'],
\end{aligned}
\end{equation*}
which mean that
$o \notin F(\partial T)$ and $W(F|\partial T,o)=\pm 1$.
\end{proof}

\subsection{Equality condition}
\label{sec:3.4}


\begin{proof}[Proof of Theorem \ref{thm:1} (equality condition)]
If $K \in \mathcal{K}(D_2)$ is a $3$-simplex, then $K$ satisfies the equality condition. 
Conversely,
asuume that $K \in \mathcal{K}(D_2)$ satisfies the equality condition, that is, $\abs{K} \abs{K^\circ} = 64/9$.
By Proposition \ref{prop:5}, there exists a sequence of convex bodies $\{K^{(n)}\}_{n \in \N} \subset \checkK(D_2)$ such that $\lim_{n \rightarrow \infty} \delta(K^{(n)}, K) = 0$.
By applying Proposition \ref{prop:1} to $K=K^{(n)}$, there exist $a_n>0$, $b_n>0$, and $c_n>0$ such that \eqref{eq:1} holds with $(a,b,c)=(a_n, b_n, c_n)$ and 
$\vv_{+-+}=\vv^{(n)}_{+-+}$, $\vv_{-++}=\vv^{(n)}_{-++}$, $\vv_{++-}=\vv^{(n)}_{++-}$, 
where we have put $\vv_{---}^{(n)} := (-a_n, -b_n, -c_n)^\mathrm{T}$ and the other as in \eqref{eq:2} for each $n \in \N$.

We note that there exists a constant $r>0$ which does not depend on $n$ such that 
$r B \subset K^{(n)} \subset r^{-1} B$.
Moreover, in the proof of Proposition \ref{prop:1}, the small positive constant $\epsilon$ depends only on $r$, and hence the triangle $T$ does not depend on $n$.
Hence, $\{(a_n, b_n, c_n)^\mathrm{T}\}_{n \in \N} \subset T$.
Since $T$ is compact, by passing to a subsequence if necessary (for simplicity, we use the same notation for the subsequences), there exists $(a, b, c)^\mathrm{T} \in T$ such that 
$(a_n, b_n, c_n)^\mathrm{T} \rightarrow (a,b,c)^\mathrm{T}$ as $n \rightarrow \infty$.
In particular, $a>0$, $b>0$, and  $c>0$.
We put $\vv_{---}:=(-a, -b, -c)^\mathrm{T}$.
In addition, normalizing $K^{(n)}$ by their dilations, we may assume that
$\vv_{---}^{(n)} = (-a_n, -b_n, -c_n)^\mathrm{T} \in \partial K^{(n)}$ for each $n$. Then, we have $\vv_{---} \in \partial K$.
Notice that, even though this normalization is done, the convergence $K^{(n)} \rightarrow K$ as $n \rightarrow \infty$ still holds.
Since
\begin{equation*}
\begin{pmatrix}
-a^\circ_n \\ -b^\circ_n \\ -c^\circ_n
\end{pmatrix}
:=
\Lambda_{K^{(n)}} \vv^{(n)}_{---}
=
\Lambda_{K^{(n)}}
\begin{pmatrix}
-a_n \\ -b_n \\ -c_n
\end{pmatrix}
\in \partial (K^{(n)})^\circ, \quad
a_n a_n^\circ + 
b_n b_n^\circ + 
c_n c_n^\circ =1,
\end{equation*}
by passing to a subsequence if necessary, there exists 
$\vv_{---}^\circ=(-a^\circ, -b^\circ, -c^\circ)^\mathrm{T} \in \partial K^\circ$ such that
\begin{equation}
\label{eq:7}
\lim_{n \rightarrow \infty} 
\begin{pmatrix}
-a^\circ_n \\ -b^\circ_n \\ -c^\circ_n
\end{pmatrix}
=
\begin{pmatrix}
-a^\circ \\ -b^\circ \\ -c^\circ
\end{pmatrix},
\quad
a a^\circ + b b^\circ + c c^\circ=1.
\end{equation}
We put 
$\vv_{+-+} := R_{13} \vv_{---}, \vv_{++-} := R_{12} \vv_{---}, \vv_{-++} := R_{23} \vv_{---} \in \partial K$ and 
$\vv_{+-+}^\circ := R_{13} \vv_{---}^\circ, \vv_{++-}^\circ := R_{12} \vv_{---}^\circ, \vv_{-++}^\circ := R_{23} \vv_{---}^\circ \in \partial K^\circ$.
Focusing on vectors $\vv_{+-+}$ and $\vv_{++-}$, we define
\begin{equation}
\label{eq:15}
\begin{aligned}
H_1&:= \Span\{\vv_{+-+}, \vv_{++-}\}, \\
L_1&:= K \cap \pos\{\vv_{+-+}, \vv_{++-}\} \subset H_1, \\
L_1^\circ&:= o*\mathcal{C}_{K^\circ}(H_1; \vv_{+-+}^\circ,\vv_{++-}^\circ) \subset H_1.
\end{aligned}
\end{equation}
Here, for the definition of $L_1^\circ$, recall the paragraph just before Remark \ref{rem:C}.
Similarly, we define $H_2$, $L_2$, $L_2^\circ$ by using 
$\vv_{++-}$ and $\vv_{-++}$, and define $H_3$, $L_3$, $L_3^\circ$ by using $\vv_{-++}$ and $\vv_{+-+}$.

Taking the estimates \eqref{eq:39}, \eqref{eq:37}, \eqref{eq:38}, and \eqref{eq:36}
into account, for $K^{(n)}$, we have
\begin{equation*}
\begin{aligned}
\frac{9}{16} |K^{(n)}| \, |(K^{(n)})^\circ|
&\geq
2 \, \overline{\mathcal{C}_{K^{(n)}}(\vv_{+-+}^{(n)}, \vv_{++-}^{(n)})}
\cdot
\overline{\Lambda^{(n)} (\mathcal{C}_{K^{(n)}}(\vv_{+-+}^{(n)}, \vv_{++-}^{(n)}))} \\
&+
2 \, \overline{\mathcal{C}_{K^{(n)}}(\vv_{++-}^{(n)}, \vv_{-++}^{(n)})}
\cdot
\overline{\Lambda^{(n)} (\mathcal{C}_{K^{(n)}}(\vv_{++-}^{(n)}, \vv_{-++}^{(n)}))} \\
&+
2 \, \overline{\mathcal{C}_{K^{(n)}}(\vv_{-++}^{(n)}, \vv_{+-+}^{(n)})}
\cdot
\overline{\Lambda^{(n)} (\mathcal{C}_{K^{(n)}}(\vv_{-++}^{(n)}, \vv_{+-+}^{(n)}))},
\end{aligned}
\end{equation*}
\begin{equation*}
\begin{aligned}
2 \, \overline{\mathcal{C}_{K^{(n)}}(\vv_{+-+}^{(n)}, \vv_{++-}^{(n)})}
\cdot
\overline{\Lambda^{(n)} (\mathcal{C}_{K^{(n)}}(\vv_{+-+}^{(n)}, \vv_{++-}^{(n)}))} 
& \geq 2 b_n b_n^\circ + 2 c_n c_n^\circ, \\
2 \, \overline{\mathcal{C}_{K^{(n)}}(\vv_{++-}^{(n)}, \vv_{-++}^{(n)})}
\cdot
\overline{\Lambda^{(n)} (\mathcal{C}_{K^{(n)}}(\vv_{++-}^{(n)}, \vv_{-++}^{(n)}))}
& \geq 2 a_n a_n^\circ + 2 c_n c_n^\circ, \\
2 \, \overline{\mathcal{C}_{K^{(n)}}(\vv_{-++}^{(n)}, \vv_{+-+}^{(n)})}
\cdot
\overline{\Lambda^{(n)} (\mathcal{C}_{K^{(n)}}(\vv_{-++}^{(n)}, \vv_{+-+}^{(n)}))} 
& \geq 2 a_n a_n^\circ + 2 b_n b_n^\circ,
\end{aligned}
\end{equation*}
and so
\begin{equation*}
\frac{9}{16} |K^{(n)}| \, |(K^{(n)})^\circ| \geq 4 (a_n a_n^\circ + b_n b_n^\circ + c_n c_n^\circ)=4.
\end{equation*}
Since $\lim_{n \rightarrow \infty} \delta(K^{(n)}, K) = 0$, it follows from the equality condition that 
\begin{equation}
\label{eq:8}
\begin{aligned}
\lim_{n \rightarrow \infty} \overline{\mathcal{C}_{K^{(n)}}(\vv_{+-+}^{(n)}, \vv_{++-}^{(n)})}
\cdot
\overline{\Lambda^{(n)} (\mathcal{C}_{K^{(n)}}(\vv_{+-+}^{(n)}, \vv_{++-}^{(n)}))} 
&= b b^\circ + c c^\circ, \\
\lim_{n \rightarrow \infty} \overline{\mathcal{C}_{K^{(n)}}(\vv_{++-}^{(n)}, \vv_{-++}^{(n)})}
\cdot
\overline{\Lambda^{(n)} (\mathcal{C}_{K^{(n)}}(\vv_{++-}^{(n)}, \vv_{-++}^{(n)}))}
& = a a^\circ + c c^\circ, \\
\lim_{n \rightarrow \infty} \overline{\mathcal{C}_{K^{(n)}}(\vv_{-++}^{(n)}, \vv_{+-+}^{(n)})}
\cdot
\overline{\Lambda^{(n)} (\mathcal{C}_{K^{(n)}}(\vv_{-++}^{(n)}, \vv_{+-+}^{(n)}))} 
& = a a^\circ + b b^\circ.
\end{aligned}
\end{equation}
Let $H_1^{(n)}$ denote the plane through the origin $o$ with the unit normal vector
\begin{equation*}
\vn^{(n)}:=\frac{\vv_{+-+}^{(n)} \times \vv_{++-}^{(n)}}{\norm{\vv_{+-+}^{(n)} \times \vv_{++-}^{(n)}}}.
\end{equation*}
Then, by the arguments in \cite{IS2}*{Section 3.3}, we have
\begin{equation*}
\begin{aligned}
\overline{\mathcal{C}_{K^{(n)}}(\vv_{+-+}^{(n)}, \vv_{++-}^{(n)})}
&=
\abs{K^{(n)} \cap \pos\{\vv_{+-+}^{(n)}, \vv_{++-}^{(n)}\}} \vn^{(n)}, \\
\overline{\Lambda^{(n)} (\mathcal{C}_{K^{(n)}}(\vv_{+-+}^{(n)}, \vv_{++-}^{(n)}))} \cdot \vn^{(n)}
&=
\abs{o*\mathcal{C}_{(K^{(n)})^\circ}(H_1^{(n)}; \Lambda^{(n)}\vv_{+-+}^{(n)}, \Lambda^{(n)} \vv_{++-}^{(n)})
}.
\end{aligned}
\end{equation*}
From \eqref{eq:7}, \eqref{eq:15}, and \eqref{eq:8},
we obtain 
\begin{equation}
\label{eq:10}
|L_1| \, |L_1^\circ| = b b^\circ + c c^\circ = \frac{1}{4} (\vv_{+-+} - \vv_{++-})\cdot(\vv^\circ_{+-+} - \vv^\circ_{++-}),
\end{equation}
since $K^{(n)} \rightarrow K$ as $n \rightarrow \infty$ in the Hausdorff distance.
By similar arguments about $L_2$ and $L_3$, we get
\begin{equation*}
|L_2| \, |L_2^\circ| = a a^\circ + c c^\circ, \quad |L_3| \, |L_3^\circ| = a a^\circ + b b^\circ.
\end{equation*}

Since 
$aa^\circ+bb^\circ+cc^\circ=1$,
$\vv_{+-+} \cdot \vv^\circ_{++-} \leq 1$, 
$\vv_{++-} \cdot \vv^\circ_{-++} \leq 1$, and
$\vv_{-++} \cdot \vv^\circ_{+-+} \leq 1$, we have
\begin{equation*}
aa^\circ \leq 1, \quad
bb^\circ \leq 1, \quad
cc^\circ \leq 1.
\end{equation*}
The case that $aa^\circ=bb^\circ=cc^\circ=1$ does not happen, since $aa^\circ+bb^\circ+cc^\circ=1$.
Moreover, without loss of generality we may assume that $aa^\circ \geq bb^\circ \geq cc^\circ$.
Thus, it suffices to consider the following three cases:
\begin{description}
\item[Case I] $aa^\circ<1$, $bb^\circ<1$, $cc^\circ<1$. 
\item[Case II] $aa^\circ=1$, $bb^\circ<1$, $cc^\circ<1$.  
\item[Case III] $aa^\circ=1$, $bb^\circ=1$, $cc^\circ<1$. 
\end{description}
We shall discuss each of these cases.

\paragraph{Case I: $aa^\circ<1$, $b b^\circ<1$, $c c^\circ<1$.}

\ Under the condition $aa^\circ+bb^\circ+cc^\circ=1$, $aa^\circ<1$ holds if and only if $\vv_{++-} \cdot \vv^\circ_{+-+}=\vv_{+-+} \cdot \vv^\circ_{++-}<1$, and so we can apply Lemma \ref{lem:4} to $L_1$ and $L_1^\circ$.
By the equality condition \eqref{eq:10}, we have
\begin{equation}
\label{eq:19}
L_1 = 
o*([\vv_{+-+}, \vv_1] \cup [\vv_1, \vv_{++-}]), \quad
L^\circ_1 = 
o*([\Pi_{H_1} \vv^\circ_{+-+}, \vv^\circ_1] \cup [\vv^\circ_1, \Pi_{H_1} \vv^\circ_{++-}]),
\end{equation}
where 
\begin{equation*}
\vv_1 = \delta_1 \vn_1 \times (\vv^\circ_{+-+}-\vv^\circ_{++-}), \quad
\vv^\circ_1 = \delta^\circ_1 \vn_1 \times (\vv_{+-+}-\vv_{++-}), \quad
\vn_1 :=
\frac{\vv_{+-+} \times \vv_{++-}}{\norm{\vv_{+-+} \times \vv_{++-}}}.
\end{equation*}
By a direct calculation, we have
\begin{equation*}
\vv_1=\frac{2\delta_1 (1-a a^\circ)}{\sqrt{b^2+c^2}}
\begin{pmatrix}
1 \\ 0 \\ 0
\end{pmatrix}, \quad
\vv^\circ_1=2\delta^\circ_1  \sqrt{b^2+c^2}
\begin{pmatrix}
1 \\ 0 \\ 0
\end{pmatrix}.
\end{equation*}
Moreover, the above constants $\delta_1$ and $\delta^\circ_1$ are expressed as follows:
\begin{equation*}
\begin{aligned}
\text{Case A\textsuperscript{\!1}: } 
\delta_1&=\frac{1}{\det(\vn_1, \vv^\circ_{+-+}, \vv^\circ_{++-})}
= \frac{\sqrt{b^2+c^2}}{2a^\circ(1-aa^\circ)}, \\
\delta^\circ_1&=\frac{\det(\vn_1, \vv^\circ_{+-+}, \vv^\circ_{++-})}{(\vv_{+-+}-\vv_{++-})\cdot (\vv^\circ_{+-+}-\vv^\circ_{++-})}
=\frac{a^\circ}{2\sqrt{b^2+c^2}}, \\
\text{Case B\textsuperscript{1}: } 
\delta_1&=\frac{\det(\vn_1,\vv_{+-+},\vv_{++-})}{(\vv_{+-+}-\vv_{++-})\cdot (\vv^\circ_{+-+}-\vv^\circ_{++-})}
= \frac{a\sqrt{b^2+c^2}}{2(1-aa^\circ)}, \\
\delta^\circ_1&=\frac{1}{\det(\vn_1,\vv_{+-+},\vv_{++-})}
=\frac{1}{2a\sqrt{b^2+c^2}}.
\end{aligned}
\end{equation*}
We note that $\delta_1 \delta^\circ_1=1/(4-4aa^\circ)$ holds in either case.
In addition, 
Case A\textsuperscript{\!1} occurs only when $\angle (\Pi_{H_1}\vv^\circ_{+-+}) o (\Pi_{H_1}\vv^\circ_{++-}) < \pi$ is satisfied, which 
is equivalent to 
$(\vn_1 \times \vv^\circ_{+-+})\cdot \vv^\circ_{++-} =2a^\circ(1-aa^\circ)/\sqrt{b^2+c^2}>0$, 
that is, $a^\circ>0$.

By the definition of $L^\circ_1$, there exists $\sigma_1 \in \R$ such that 
\begin{equation*}
\hat{\vv}^\circ_1 := \delta^\circ_1 \vn_1 \times (\vv_{+-+}-\vv_{++-}) + \sigma_1 \vn_1 \in K^\circ.
\end{equation*}
By the symmetry and convexity of $K^\circ$, we have
\begin{equation*}
R_{23} \hat{\vv}^\circ_1 = \delta^\circ_1 \vn_1 \times (\vv_{+-+}-\vv_{++-}) - \sigma_1 \vn_1 \in K^\circ, \quad 
\vv^\circ_1 =
\frac{\hat{\vv}^\circ_1 + R_{23} \hat{\vv}^\circ_1}{2} \in K^\circ.
\end{equation*}
Hence, \eqref{eq:19} asserts that 
\begin{equation}
\label{eq:18}
[\vv_{+-+}, \vv_1], \,
[\vv_1, \vv_{++-}] \subset \partial K, \quad
[\vv^\circ_{+-+}, \vv^\circ_1], \,
[\vv^\circ_1, \vv^\circ_{++-}] \subset \partial K^\circ.
\end{equation}
By similar arguments about $L_2$ and $L_3$, we can define 
$\vv_2, \vv_3 \in \partial K$ and $\vv^\circ_2, \vv^\circ_3 \in \partial K^\circ$ so as to satisfy
\begin{equation*}
\begin{aligned}
&[\vv_{++-}, \vv_2], \,
[\vv_2, \vv_{-++}] \subset \partial K, \quad
[\vv^\circ_{++-}, \vv^\circ_2], \,
[\vv^\circ_2, \vv^\circ_{-++}] \subset \partial K^\circ, \\
&[\vv_{-++}, \vv_3], \,
[\vv_3, \vv_{+-+}] \subset \partial K, \quad
[\vv^\circ_{-++}, \vv^\circ_3], \,
[\vv^\circ_3, \vv^\circ_{+-+}] \subset \partial K^\circ.
\end{aligned}
\end{equation*}
Next, by using these segments and the $D_2$-symmetry of $K$, we divide $K$ and $K^\circ$,
and estimate the volume product of $K$.
By \eqref{eq:18} and the above conditions about $\vv_2$ and $\vv_3$, we have
\begin{equation}
\label{eq:20}
\begin{aligned}
\mathcal{C}&:= 
\mathcal{C}_{K}(\vv_{+-+}, \vv_{1}, \vv_{++-}, \vv_{2}, \vv_{-++}, \vv_{3}, \vv_{+-+}) \\
&=
[\vv_{+-+}, \vv_{1}] \cup
[\vv_{1}, \vv_{++-}] \cup
[\vv_{++-}, \vv_{2}] \cup
[\vv_{2}, \vv_{-++}] \cup
[\vv_{-++}, \vv_{3}] \cup
[\vv_{3}, \vv_{+-+}], \\
\mathcal{C}^\circ&:= 
\mathcal{C}_{K^\circ}(\vv^\circ_{+-+}, \vv^\circ_{1}, \vv^\circ_{++-}, \vv^\circ_{2}, \vv^\circ_{-++}, \vv^\circ_{3}, \vv^\circ_{+-+}) \\
&=
[\vv^\circ_{+-+}, \vv^\circ_{1}] \cup
[\vv^\circ_{1}, \vv^\circ_{++-}] \cup
[\vv^\circ_{++-}, \vv^\circ_{2}] \cup
[\vv^\circ_{2}, \vv^\circ_{-++}] \cup
[\vv^\circ_{-++}, \vv^\circ_{3}] \cup
[\vv^\circ_{3}, \vv^\circ_{+-+}].
\end{aligned}
\end{equation}
Using them, we put
\begin{equation*}
K_0 := o*\mathcal{S}_K(\mathcal{C}), \quad
K^\circ_0 := o*\mathcal{S}_{K^\circ}(\mathcal{C}^\circ).
\end{equation*}
Clearly, $\mathcal{C}$ is a simple closed curve on $\partial K$.
Moreover, $\mathcal{C}^\circ$ is also a simple closed curve on $\partial K^\circ$.
Indeed, for $\vx^\circ:= (1/a,1/b,1/c)^\mathrm{T}$
we can easily check that
$(\vv^\circ_{+-+} \times \vv^\circ_1) \cdot \vx^{\circ}$, 
$(\vv^\circ_1 \times \vv^\circ_{++-}) \cdot \vx^{\circ}$, 
$(\vv^\circ_{++-} \times \vv^\circ_2) \cdot \vx^{\circ}$, 
$(\vv^\circ_2 \times \vv^\circ_{-++}) \cdot \vx^{\circ}$, 
$(\vv^\circ_{-++} \times \vv^\circ_3) \cdot \vx^{\circ}$, and
$(\vv^\circ_3 \times \vv^\circ_{+-+}) \cdot \vx^{\circ}$ are nonnegative.
Thus $\mathcal{C}^\circ$ is a simple closed curve on $\partial K^\circ$ surrounding the point $\vx^\circ/\mu_{K^\circ}(\vx^\circ) \in \partial K^\circ$.
By the symmetry, $\abs{K} \abs{K^\circ} = 16 \abs{K_0} \abs{K^\circ_0}$ holds.
Applying Lemma \ref{lem:7} to $\mathcal{C}$ and $\mathcal{C}^\circ$, and using the equality condition, we get 
\begin{equation*}
\frac{4}{9} = \abs{K_0} \abs{K^\circ_0}
\geq \frac{1}{9}\overline{\mathcal{C}} \cdot \overline{\mathcal{C}^\circ}.
\end{equation*}
Here, by \eqref{eq:20}, we have
\begin{equation*}
\begin{aligned}
2\overline{\mathcal{C}} &=
\vv_{+-+} \times \vv_{1} +
\vv_{1} \times \vv_{++-} +
\vv_{++-} \times \vv_{2} +
\vv_{2} \times \vv_{-++} +
\vv_{-++} \times \vv_{3} +
\vv_{3} \times \vv_{+-+} , \\
2\overline{\mathcal{C}^\circ} &=
\vv^\circ_{+-+} \times \vv^\circ_{1} +
\vv^\circ_{1} \times \vv^\circ_{++-} +
\vv^\circ_{++-} \times \vv^\circ_{2} +
\vv^\circ_{2} \times \vv^\circ_{-++} +
\vv^\circ_{-++} \times \vv^\circ_{3} +
\vv^\circ_{3} \times \vv^\circ_{+-+},
\end{aligned}
\end{equation*}
and
\begin{equation*}
\begin{aligned}
\vv_{+-+} \times \vv_{1} +
\vv_{1} \times \vv_{++-}
&=\frac{2 \delta_1 (1-a a^\circ)}{a \sqrt{b^2+c^2}}
\begin{pmatrix}
0 \\ 2ac \\ 2ab
\end{pmatrix}, \\
\vv^\circ_{+-+} \times \vv^\circ_{1} +
\vv^\circ_{1} \times \vv^\circ_{++-} &=
\frac{2 \delta^\circ_1 \sqrt{b^2+c^2}}{a^\circ}
\begin{pmatrix}
0 \\ 2a^\circ c^\circ \\ 2a^\circ b^\circ
\end{pmatrix}.
\end{aligned}
\end{equation*}
Here, we denote the coefficients above by $\epsilon_1$ and $\epsilon_1^\circ$ as follows:
\begin{equation*}
\epsilon_1 := 
\frac{2 \delta_1 (1-a a^\circ)}{a \sqrt{b^2+c^2}}, \quad 
\epsilon_1^\circ := 
\frac{2 \delta^\circ_1 \sqrt{b^2+c^2}}{a^\circ}.
\end{equation*}
In a similar way, we define constants $\epsilon_i$, $\epsilon^\circ_i$ $(i=2,3)$.
We note that $\epsilon_1 \epsilon^\circ_1=1/(aa^\circ)$, $\epsilon_2 \epsilon^\circ_2=1/(bb^\circ)$, and $\epsilon_3 \epsilon^\circ_3=1/(cc^\circ)$.
Recall that \eqref{eq:1} holds for $K^{(n)}$, so that  as $n \rightarrow \infty$,  \eqref{eq:1} also holds for $K$.
Since 
\begin{equation*}
\begin{aligned}
&\overline{\mathcal{C}_K(\vv_{+-+}, \vv_{++-})} = 
\frac{1}{2} 
\left(
\vv_{+-+} \times \vv_{1}
+
\vv_{1} \times \vv_{++-} 
\right)
=\epsilon_1
\begin{pmatrix}
0 \\ ac \\ ab
\end{pmatrix}, \\
&\overline{\mathcal{C}_K(\vv_{++-}, \vv_{-++})}
= 
\epsilon_2
\begin{pmatrix}
bc \\ 0 \\ ab
\end{pmatrix}, \quad
\overline{\mathcal{C}_K(\vv_{-++}, \vv_{+-+})}
= 
\epsilon_3
\begin{pmatrix}
bc \\ ac \\ 0
\end{pmatrix},
\end{aligned}
\end{equation*}
the condition \eqref{eq:1} for $K$ means that
\begin{equation}
\label{eq:28}
\epsilon:=
\epsilon_1 =
\epsilon_2 =
\epsilon_3.
\end{equation}
Consequently, we obtain
\begin{equation*}
\overline{\mathcal{C}} = 
\begin{pmatrix}
bc(\epsilon_2 + \epsilon_3) \\
ac(\epsilon_1 + \epsilon_3) \\
ab(\epsilon_1 + \epsilon_2)
\end{pmatrix}
=2 \epsilon
\begin{pmatrix}
bc \\
ac \\
ab
\end{pmatrix}, \quad
\overline{\mathcal{C}^\circ}
=
\epsilon^\circ_1
\begin{pmatrix}
0 \\ a^\circ c^\circ \\ a^\circ b^\circ
\end{pmatrix}
+
\epsilon^\circ_2
\begin{pmatrix}
b^\circ c^\circ \\ 0 \\ a^\circ b^\circ
\end{pmatrix}
+
\epsilon^\circ_3
\begin{pmatrix}
b^\circ c^\circ \\ a^\circ c^\circ \\ 0
\end{pmatrix}
=
\frac{1}{\epsilon}
\begin{pmatrix}
\displaystyle \frac{cc^\circ + bb^\circ}{bc}\\
\displaystyle \frac{cc^\circ + aa^\circ}{ac} \\
\displaystyle \frac{bb^\circ+aa^\circ}{ab}
\end{pmatrix},
\end{equation*}
which yield
\begin{equation}
\label{eq:9}
4 = 9 \abs{K_0} \abs{K^\circ_0} \geq 
\overline{\mathcal{C}} \cdot \overline{\mathcal{C}^\circ} =
4(a a^\circ + b b^\circ + c c^\circ) =4.
\end{equation}
Hence, the equality condition in Lemma \ref{lem:7} is satisfied, and there exist $\overline{\vv}, \overline{\vv}^\circ \in \R^3$ and $\tau, \tau^\circ>0$ such that
\begin{equation}
\label{eq:11}
K_0=o * (\overline{\vv}*\mathcal{C}), \quad
K^\circ_0=o * (\overline{\vv}^\circ *\mathcal{C}^\circ), \quad
\overline{\vv} = \tau \overline{\mathcal{C}^\circ}, \quad
\overline{\vv}^\circ = \tau^\circ \overline{\mathcal{C}}.
\end{equation}
Since $\mathcal{C}$ is a closed polygonal line on $\partial K$ and 
\begin{equation*}
\begin{aligned}
\partial K &= S_K(\mathcal{C})
\cup R_{23} S_K(\mathcal{C})
\cup R_{13} S_K(\mathcal{C})
\cup R_{12} S_K(\mathcal{C}) \\
&=
\left(\overline{\vv} * \mathcal{C}\right)
\cup \left(R_{23} \overline{\vv} * R_{23}\mathcal{C}\right)
\cup \left(R_{13} \overline{\vv} * R_{13}\mathcal{C}\right)
\cup \left(R_{12} \overline{\vv} * R_{12}\mathcal{C}\right)
\end{aligned}
\end{equation*}
by the $D_2$-symmetry, $K$ is a polyhedron.
In addition, each face of $K$ contains at least one of the four points
$\overline{\vv}$, 
$R_{23}\overline{\vv}$, 
$R_{13}\overline{\vv}$, and
$R_{12}\overline{\vv}$.
Similarly, $K^\circ$ is also a polyhedron.

Finally, we shall determine the shape of $K$ (or $K^\circ$).
The possibility of the constants $\epsilon_i$ $(i=1,2,3)$ summarizes as follows:
\begin{equation*}
\epsilon_1 =
\begin{cases}
1/a a^\circ & \text{ if Case A\textsuperscript{\!1}}, \\
1 & \text{ if Case B\textsuperscript{1}}, 
\end{cases} \quad
\epsilon_2 =
\begin{cases}
1/b b^\circ & \text{ if Case A\textsuperscript{\!2}}, \\
1 & \text{ if Case B\textsuperscript{2}}, 
\end{cases} \quad 
\epsilon_3 =
\begin{cases}
1/c c^\circ & \text{ if Case A\textsuperscript{\!3}}, \\
1 & \text{ if Case B\textsuperscript{3}}.
\end{cases}
\end{equation*}
However, it follows from \eqref{eq:28} that only two cases,
Case A\textsuperscript{\!1}--Case A\textsuperscript{\!2}--Case A\textsuperscript{\!3} or 
Case B\textsuperscript{1}--Case B\textsuperscript{2}--Case B\textsuperscript{3}, can occur.
Indeed, suppose that Case A\textsuperscript{\!1} and Case B\textsuperscript{2} hold, then
we have $\epsilon_1=1/(a a^\circ)$ and $\epsilon_2=1$, which contradict to \eqref{eq:28}
since $a a^\circ <1$.
Similarly, Case A\textsuperscript{\!$i$} and Case B\textsuperscript{$j$} cannot occur simultaneously. 

Now, we consider Case A\textsuperscript{\!1}--Case A\textsuperscript{\!2}--Case A\textsuperscript{\!3}.
Since $aa^\circ+bb^\circ+cc^\circ=1$, by \eqref{eq:28}, we have
\begin{equation*}
aa^\circ=bb^\circ=cc^\circ=\frac{1}{3}, \quad
\epsilon=3, \quad
a^\circ, b^\circ, c^\circ>0.
\end{equation*}
By \eqref{eq:11}, it holds that
\begin{equation}
\label{eq:3}
K_0=o*(\overline{\vv}*(
[\vv_{+-+}, \vv_{1}] \cup
[\vv_{1}, \vv_{++-}] \cup
[\vv_{++-}, \vv_{2}] \cup
[\vv_{2}, \vv_{-++}] \cup
[\vv_{-++}, \vv_{3}] \cup
[\vv_{3}, \vv_{+-+}])).
\end{equation}
By a direct calculation, we get
\begin{equation*}
\vv_1=
\begin{pmatrix}
1/a^\circ \\ 0 \\ 0
\end{pmatrix}, \quad
\vv_2=
\begin{pmatrix}
0 \\ 1/b^\circ \\ 0
\end{pmatrix}, \quad
\vv_3=
\begin{pmatrix}
0 \\ 0 \\ 1/c^\circ
\end{pmatrix}.
\end{equation*}
It follows from
\begin{equation*}
\vv_2 \cdot \vv_{-++} = \frac{b}{b^\circ}>0, \quad
\vv_3 \cdot \vv_{-++} = \frac{c}{c^\circ}>0, \quad
\det(\vv_2, \vv_{-++}, \vv_3) = \frac{a}{b^\circ c^\circ}>0
\end{equation*}
that $\conv\{\vv_2, \vv_{-++}, \vv_3\} \subset K_0$ and the three points $\vv_2$, $\vv_{-++}$, $\vv_3$ do not lie on a straight line.
Moreover, since $\vv_{2} \cdot \vv^\circ_{-++}= \vv_{-++}\cdot \vv^\circ_{-++}= \vv_3\cdot \vv^\circ_{-++}=1$,
$\vv^\circ_{-++}$ is a vertex of $K^\circ$ and its polar dual is a face of $K$.
We denote the face by $F$, then $\conv\{\vv_2, \vv_{-++}, \vv_3\}$ is a subset of $F$ and
is contained in $K_0 \cap \partial K$, and all faces of $K_0 \cap \partial K$ contain $\overline{\vv}$ from \eqref{eq:3}. Thus, $\overline{\vv} \in F$, and so $\overline{\vv} \cdot \vv^\circ_{-++}=1$.
By \eqref{eq:11}, we have
\begin{equation*}
1 = \overline{\vv} \cdot \vv^\circ_{-++}
=
\frac{2\tau}{9}
\begin{pmatrix}
1/bc \\
1/ac \\
1/ab
\end{pmatrix}
\cdot \vv^\circ_{-++}
=\frac{2\tau}{27abc}, \quad
\overline{\vv}
=
3
\begin{pmatrix}
a \\ b \\ c
\end{pmatrix}
=
\begin{pmatrix}
1/a^\circ \\ 1/b^\circ \\ 1/c^\circ
\end{pmatrix}.
\end{equation*}
Thus, we obtain 
\begin{equation}
\label{eq:42}
\begin{aligned}
\abs{K} &\geq \abs{\conv\{\overline{\vv}, R_{23}\overline{\vv}, R_{13}\overline{\vv}, R_{12}\overline{\vv}\}} = \frac{8}{3a^\circ b^\circ c^\circ}, \\
\abs{K^\circ} &\geq \abs{\conv\{\vv^\circ_{-++}, \vv^\circ_{+-+}, \vv^\circ_{++-}, \vv^\circ_{---}\}} = \frac{8a^\circ b^\circ c^\circ}{3}.
\end{aligned}
\end{equation}
By the equality condtion, \eqref{eq:42} implies that
\begin{equation*}
K=\conv\{\overline{\vv}, R_{23}\overline{\vv}, R_{13}\overline{\vv}, R_{12}\overline{\vv}\}, \quad
K^\circ=\conv\{\vv^\circ_{-++}, \vv^\circ_{+-+}, \vv^\circ_{++-}, \vv^\circ_{---}\},
\end{equation*}
hence, $K$ and $K^\circ$ are $3$-simplices.
In Case B\textsuperscript{1}--Case B\textsuperscript{2}--Case B\textsuperscript{3}, 
we exchange $K$ for $K^\circ$ and can discuss in the same way.
Then, we can conclude that $K$ and $K^\circ$ are $3$-simplices.

\paragraph{Case II: $aa^\circ=1$, $b b^\circ<1$, $cc^\circ<1$.}

\ In this case, we note that $bb^\circ+cc^\circ=0$ since $aa^\circ+bb^\circ+cc^\circ=1$, and
we have
\begin{equation*}
\vv_{+-+} \cdot \vv^\circ_{+-+}
=
\vv_{+-+} \cdot \vv^\circ_{++-}
=
\vv_{++-} \cdot \vv^\circ_{+-+}
=
\vv_{++-} \cdot \vv^\circ_{++-}
=1.
\end{equation*}
Hence, the inner product of every point on $[\vv_{+-+}, \vv_{++-}]$ and one on $[\vv^\circ_{+-+}, \vv^\circ_{++-}]$
is always equal to one, and so
$[\vv_{+-+}, \vv_{++-}] \subset \partial K$ and 
$[\vv^\circ_{+-+}, \vv^\circ_{++-}] \subset \partial K^\circ$ hold.

As we discussed in Case I, since $b b^\circ<1$ and $c c^\circ<1$, we can characterize 
$L_2$, $L^\circ_2$, $L_3$, $L^\circ_3$ by Lemma \ref{lem:4}.
From the equality conditions and the above facts, we obtain
\begin{equation*}
\begin{aligned}
\mathcal{C}&:= 
\mathcal{C}_{K}(\vv_{+-+}, \vv_{++-}, \vv_{2}, \vv_{-++}, \vv_{3}, \vv_{+-+}) \\
&=
[\vv_{+-+}, \vv_{++-}] \cup
[\vv_{++-}, \vv_{2}] \cup
[\vv_{2}, \vv_{-++}] \cup
[\vv_{-++}, \vv_{3}] \cup
[\vv_{3}, \vv_{+-+}], \\
\mathcal{C}^\circ&:= 
\mathcal{C}_{K^\circ}(\vv^\circ_{+-+}, \vv^\circ_{++-}, \vv^\circ_{2}, \vv^\circ_{-++}, \vv^\circ_{3}, \vv^\circ_{+-+}) \\
&=
[\vv^\circ_{+-+}, \vv^\circ_{++-}] \cup
[\vv^\circ_{++-}, \vv^\circ_{2}] \cup
[\vv^\circ_{2}, \vv^\circ_{-++}] \cup
[\vv^\circ_{-++}, \vv^\circ_{3}] \cup
[\vv^\circ_{3}, \vv^\circ_{+-+}]
\end{aligned}
\end{equation*}
and 
\begin{equation*}
\begin{aligned}
2\overline{\mathcal{C}} &=
\vv_{+-+} \times \vv_{++-} +
\vv_{++-} \times \vv_{2} +
\vv_{2} \times \vv_{-++} +
\vv_{-++} \times \vv_{3} +
\vv_{3} \times \vv_{+-+} , \\
2\overline{\mathcal{C}^\circ} &=
\vv^\circ_{+-+} \times \vv^\circ_{++-} +
\vv^\circ_{++-} \times \vv^\circ_{2} +
\vv^\circ_{2} \times \vv^\circ_{-++} +
\vv^\circ_{-++} \times \vv^\circ_{3} +
\vv^\circ_{3} \times \vv^\circ_{+-+}.
\end{aligned}
\end{equation*}
It follows from
\begin{equation*}
\begin{aligned}
&\overline{\mathcal{C}_K(\vv_{+-+}, \vv_{++-})} = 
\frac{1}{2} \vv_{+-+} \times \vv_{++-}
=
\begin{pmatrix}
0 \\ ac \\ ab
\end{pmatrix}, \\
&\overline{\mathcal{C}_K(\vv_{++-}, \vv_{-++})}
= 
\epsilon_2
\begin{pmatrix}
bc \\ 0 \\ ab
\end{pmatrix}, \quad
\overline{\mathcal{C}_K(\vv_{-++}, \vv_{+-+})}
= 
\epsilon_3
\begin{pmatrix}
bc \\ ac \\ 0
\end{pmatrix},
\end{aligned}
\end{equation*}
and the condition \eqref{eq:1} for $K$ that
\begin{equation}
\label{eq:29}
1=\epsilon_2 =\epsilon_3
\end{equation}
and 
\begin{equation*}
\overline{\mathcal{C}} = 
2 
\begin{pmatrix}
bc \\
ac \\
ab
\end{pmatrix}, \quad
\overline{\mathcal{C}^\circ} = 
\begin{pmatrix}
\displaystyle \frac{cc^\circ + bb^\circ}{bc}\\
\displaystyle \frac{cc^\circ + aa^\circ}{ac} \\
\displaystyle \frac{bb^\circ+aa^\circ}{ab}
\end{pmatrix}.
\end{equation*}
Then, we obtain
\begin{equation*}
4=9 \abs{K_0} \abs{K^\circ_0} \geq 
\overline{\mathcal{C}} \cdot \overline{\mathcal{C}^\circ} =
4(a a^\circ + b b^\circ + c c^\circ)
=4.
\end{equation*}
Moreover, by \eqref{eq:29}, only 
Case B\textsuperscript{2} and Case B\textsuperscript{3} can occur. Hence we have
\begin{equation*}
\begin{aligned}
&K^\circ_0=o*(\overline{\vv}^\circ*(
[\vv^\circ_{+-+}, \vv^\circ_{++-}] \cup
[\vv^\circ_{++-}, \vv^\circ_{2}] \cup
[\vv^\circ_{2}, \vv^\circ_{-++}] \cup
[\vv^\circ_{-++}, \vv^\circ_{3}] \cup
[\vv^\circ_{3}, \vv^\circ_{+-+}])), \\
&\vv^\circ_2=
\begin{pmatrix}
0 \\ 1/b \\ 0
\end{pmatrix}, \quad
\vv^\circ_3=
\begin{pmatrix}
0 \\ 0 \\ 1/c
\end{pmatrix}.
\end{aligned}
\end{equation*}
Similarly as in Case I, we get $\vv_{-++} \cdot \overline{\vv}^\circ=1$ and
\begin{equation*}
\overline{\vv}^\circ=
\begin{pmatrix}
1/a \\ 1/b \\ 1/c
\end{pmatrix}.
\end{equation*}
Therefore, combining
\begin{equation*}
\begin{aligned}
\abs{K} &\geq \abs{\conv\{\vv_{-++}, \vv_{+-+}, \vv_{++-}, \vv_{---}\}} = \frac{8a b c}{3}, \\
\abs{K^\circ} &\geq \abs{\conv\{\overline{\vv}^\circ, R_{23}\overline{\vv}^\circ, R_{13}\overline{\vv}^\circ, R_{12}\overline{\vv}^\circ\}} = \frac{8}{3abc}
\end{aligned}
\end{equation*}
with the equality condition, we conclude that $K$ and $K^\circ$ are $3$-simplices.

\paragraph{Case III: $aa^\circ=b b^\circ=1$, $c c^\circ=-1$.}

By a direct calculation, we have
\begin{equation*}
\begin{aligned}
\abs{K} &\geq \abs{\conv\{\vv_{-++}, \vv_{+-+}, \vv_{++-}, \vv_{---}\}} = \frac{8abc}{3}, \\
\abs{K^\circ} &\geq \abs{\conv\{\vv^\circ_{-++}, \vv^\circ_{+-+}, \vv^\circ_{++-}, \vv^\circ_{---}\}} = \frac{8\abs{a^\circ b^\circ c^\circ}}{3}.
\end{aligned}
\end{equation*}
Since $\abs{a a^\circ b b^\circ c c^\circ}=1$, the equality condition implies that
\begin{equation*}
K=\conv\{\vv_{-++}, \vv_{+-+}, \vv_{++-}, \vv_{---}\}, \quad
K^\circ=\conv\{\vv^\circ_{-++}, \vv^\circ_{+-+}, \vv^\circ_{++-}, \vv^\circ_{---}\},
\end{equation*}
hence $K$ and $K^\circ$ are $3$-simplices.
\end{proof}

\section{The case with $S_4$-symmetry}
\label{sec:4}

\subsection{Settings and preliminary calculations}
\label{sec:4.1}

We put
\begin{equation*}
g:= 
\begin{pmatrix}
0 & -1 & 0 \\
1 & 0 & 0 \\
0 & 0 & -1
\end{pmatrix}
\, (=R_4 H)
, \quad
\vv_0:=
\begin{pmatrix}
1 \\ 0 \\ u
\end{pmatrix}, \quad
\vv_k:= g^k \vv_0
=
\begin{pmatrix}
\cos (k\pi/2) \\
\sin (k\pi/2) \\
(-1)^k u
\end{pmatrix} \quad
(k \in \Z),
\end{equation*}
where $u$ is a positive parameter which will be determined later (see Proposition \ref{prop:4}).
Then, without loss of generality we may assume that the group $S_4$ is generated by $g$, and $g^4=E$ holds.
Let $K \in \checkK(S_4)$ and put
\begin{equation*}
\begin{aligned}
K_0 &:= o*\mathcal{S}_K(\vv_0, \vv_1, \vv_2)=K \cap \pos(\vv_0, \vv_1, \vv_2), \\
K^\circ_0 & :=o*\Lambda_K(\mathcal{S}_K(\vv_0, \vv_1, \vv_2)).
\end{aligned}
\end{equation*}
Then, we see that 
\begin{equation*}
\abs{K}=4  \abs{K_0}, \quad
\abs{K^\circ}=4  \abs{K_0^\circ}, \quad
\abs{K} \, \abs{K^\circ} = 16 \abs{K_0} \, \abs{K_0^\circ}.
\end{equation*}
By the volume estimate (Lemma \ref{lem:1}), we have
\begin{equation}
\label{eq:50}
\begin{aligned}
& 9 \abs{K_0} \, \abs{K_0^\circ} \\
& \geq 
\left(
\overline{\mathcal{C}(\vv_0, \vv_1)}
+
\overline{\mathcal{C}(\vv_1, \vv_2)}
+
\overline{\mathcal{C}(\vv_2, \vv_0)}
\right)
\cdot 
\left(
\overline{\Lambda (\mathcal{C}(\vv_0, \vv_1))}
+
\overline{\Lambda (\mathcal{C}(\vv_1, \vv_2))}
+
\overline{\Lambda (\mathcal{C}(\vv_2, \vv_0))}
\right).
\end{aligned}
\end{equation}
We write $(*)$ for the right-hand side of \eqref{eq:50} for simplicity, and put
\begin{equation*}
\begin{pmatrix}
a \\ b \\ c
\end{pmatrix}
:=
\overline{\mathcal{C}(\vv_0, \vv_1)}, \quad
\begin{pmatrix}
a^\circ \\ b^\circ \\ c^\circ
\end{pmatrix}
:=
\overline{\Lambda (\mathcal{C}(\vv_0, \vv_1))}.
\end{equation*}
Since
\begin{equation*}
\overline{\mathcal{C}(\vv_0, \vv_1)} 
=\norm{\overline{\mathcal{C}(\vv_0, \vv_1)}} \frac{\vv_0 \times \vv_1}{\norm{\vv_0 \times \vv_1}}
=\frac{\abs{o * \mathcal{C}(\vv_0, \vv_1)}}{\sqrt{1+2u^2}}
\begin{pmatrix}
-u \\ u \\ 1
\end{pmatrix},
\end{equation*}
we have
\begin{equation*}
a=-uc, \quad b=uc, \quad
c=\frac{|o*\mathcal{C}(\vv_0,\vv_1)|}{\sqrt{1+2 u^2}}>0.
\end{equation*}
Thus, we get
\begin{equation*}
\overline{\mathcal{C}(\vv_0, \vv_1)}
\cdot
\overline{\Lambda (\mathcal{C}(\vv_0, \vv_1))}
=c (-u a^\circ+u b^\circ+c^\circ).
\end{equation*}
In addition, since $\det g=-1$, we have
\begin{equation*}
\overline{\mathcal{C}(\vv_1, \vv_2)}
=
\overline{\mathcal{C}(g \vv_0, g \vv_1)}
=
-g
\overline{\mathcal{C}(\vv_0, \vv_1)}, \quad
\overline{\Lambda (\mathcal{C}(\vv_1, \vv_2))}
=
-g
\overline{\Lambda (\mathcal{C}(\vv_0, \vv_1))}.
\end{equation*}
Next, we put
\begin{equation*}
\begin{pmatrix}
\alpha \\ \beta \\ \gamma
\end{pmatrix}
:=
\overline{\mathcal{C}(\vv_2, \vv_0)}, \quad
\begin{pmatrix}
\alpha^\circ \\ \beta^\circ \\ \gamma^\circ
\end{pmatrix}
:=
\overline{\Lambda (\mathcal{C}(\vv_2, \vv_0))}.
\end{equation*}
Then, since
\begin{equation*}
\overline{\mathcal{C}(\vv_2, \vv_0)} 
=\norm{\overline{\mathcal{C}(\vv_2, \vv_0)}} \frac{\vv_2 \times \vv_0}{\norm{\vv_2 \times \vv_0}}
= \abs{o*\mathcal{C}(\vv_2, \vv_0)}
\begin{pmatrix}
0 \\ 1 \\ 0
\end{pmatrix},
\end{equation*}
we have
\begin{equation*}
\alpha=\gamma=0, \quad
\beta =\abs{o*\mathcal{C}(\vv_2, \vv_0)}>0.
\end{equation*}
Since
\begin{equation*}
g^2 
\overline{\Lambda (\mathcal{C}(\vv_2, \vv_0))} =
\overline{\Lambda (\mathcal{C}(g^2 \vv_2, g^2 \vv_0))} =
\overline{\Lambda (\mathcal{C}(\vv_0, \vv_2))} =
-
\overline{\Lambda (\mathcal{C}(\vv_2, \vv_0))},
\end{equation*}
we obtain
\begin{equation*}
g^2 
\begin{pmatrix}
\alpha^\circ \\ \beta^\circ \\ \gamma^\circ
\end{pmatrix}
=
\begin{pmatrix}
-\alpha^\circ \\ -\beta^\circ \\ \gamma^\circ
\end{pmatrix}
=
-
\begin{pmatrix}
\alpha^\circ \\ \beta^\circ \\ \gamma^\circ
\end{pmatrix},
\end{equation*}
which implies $\gamma^\circ=0$. Thus, we have
\begin{equation*}
\overline{\mathcal{C}(\vv_2, \vv_0)}
\cdot
\overline{\Lambda (\mathcal{C}(\vv_2, \vv_0))}
=\beta \beta^\circ.
\end{equation*}
Therefore, we calculate $(*)$ as follows:
\begin{equation}
\label{eq:51}
\begin{aligned}
(*) 
&=
\left(
\overline{\mathcal{C}(\vv_0, \vv_1)}
+
\overline{\mathcal{C}(\vv_1, \vv_2)}
+
\overline{\mathcal{C}(\vv_2, \vv_0)}
\right)
\cdot 
\left(
\overline{\Lambda (\mathcal{C}(\vv_0, \vv_1))}
+
\overline{\Lambda (\mathcal{C}(\vv_1, \vv_2)}
+
\overline{\Lambda (\mathcal{C}(\vv_2, \vv_0))}
\right) \\
&=
\left(
(I-g) \overline{\mathcal{C}(\vv_0, \vv_1)}
+ \overline{\mathcal{C}(\vv_2, \vv_0)}
\right)
\cdot 
\left(
(I-g) \overline{\Lambda (\mathcal{C}(\vv_0, \vv_1))}
+ \overline{\Lambda (\mathcal{C}(\vv_2, \vv_0))}
\right)
\\
&=
\left(
\begin{pmatrix}
1 & 1 & 0 \\
-1 & 1 & 0 \\
0 & 0 & 2
\end{pmatrix}
\begin{pmatrix}
-uc \\ uc \\ c
\end{pmatrix}
+
\begin{pmatrix}
0 \\ \beta \\ 0
\end{pmatrix}
\right)
\cdot
\left(
\begin{pmatrix}
1 & 1 & 0 \\
-1 & 1 & 0 \\
0 & 0 & 2
\end{pmatrix}
\begin{pmatrix}
a^\circ \\ b^\circ \\ c^\circ
\end{pmatrix}
+
\begin{pmatrix}
\alpha^\circ \\ \beta^\circ \\ 0
\end{pmatrix}
\right) \\
&=
\begin{pmatrix}
0 \\ 2uc + \beta \\ 2c
\end{pmatrix}
\cdot
\begin{pmatrix}
a^\circ+b^\circ + \alpha^\circ \\
-a^\circ+b^\circ + \beta^\circ \\
2 c^\circ 
\end{pmatrix}
=
(2uc+\beta)(-a^\circ+b^\circ+\beta^\circ)
+4c c^\circ.
\end{aligned}
\end{equation}
Here, by Proposition \ref{prop:4} below we can choose $u>0$ such that
$\beta = 2uc$, that is,
\begin{equation}
\label{eq:4}
\abs{o*\mathcal{C}_K(\vv_2,\vv_0)} = 
\frac{2 u \abs{o*\mathcal{C}_K(\vv_0,\vv_1)}}{\sqrt{1+2 u^2}}.
\end{equation}
\begin{proposition}
\label{prop:4}
For any $K \in \checkK(S_4)$, there exists a constant $u>0$ such that the condition \eqref{eq:4} holds.
\end{proposition}
\begin{proof}
We consider two functions of one variable $u$ defined by
\begin{equation*}
c(u):=\frac{\abs{o*\mathcal{C}_K(\vv_0, \vv_1)}}{\sqrt{1+2 u^2}}, \quad
\beta(u):=\abs{o*\mathcal{C}_{K}(\vv_2, \vv_0)}.
\end{equation*}
Then, we note that $c(u)$ and $\beta(u)$ depend on $u$ continuously, since $\vv_0$ depends on $u$ continuously.

Since $B \subset \R^3$ is the unit closed ball centered at $o$, as in Section \ref{sec:3.3}, we have
\begin{equation*}
\begin{aligned}
\abs{o*\mathcal{C}_{rB}(\vv_0,\vv_1)}
&=
\frac{r^2}{2}
\arccos \left(
\frac{\vv_0 \cdot \vv_1}{\norm{\vv_0} \, \norm{\vv_1}}
\right)
=
\frac{r^2}{2}
\arccos \left(
\frac{-u^2}{1+u^2}
\right), \\
\abs{o*\mathcal{C}_{rB}(\vv_2,\vv_0)}
&=
\frac{r^2}{2}
\arccos \left(
\frac{\vv_2 \cdot \vv_0}{\norm{\vv_2} \, \norm{\vv_0}}
\right)
=
\frac{r^2}{2}
\arccos \left(
\frac{-1 + u^2}{1+u^2}
\right)
\end{aligned}
\end{equation*}
for any  $r>0$.
Now we take the radius $r>0$ such that 
$rB \subset K \subset r^{-1}B$.
Then, we obtain
\begin{equation*}
\begin{aligned}
\frac{r^2}{2\sqrt{1+2u^2}}
\arccos \left(
\frac{ - u^2}{1+u^2}
\right)
&\leq c(u) \leq 
\frac{1}{2 r^2\sqrt{1+2u^2}}
\arccos \left(
\frac{ - u^2}{1+u^2}
\right), \\
\frac{r^2}{2}
\arccos \left(
\frac{-1 + u^2}{1+u^2}
\right)
& \leq \beta(u)
\leq
\frac{1}{2r^2}
\arccos \left(
\frac{-1 + u^2}{1+u^2}
\right).
\end{aligned}
\end{equation*}
Hence
\begin{equation}
\label{eq:30}
f_1(u)
\leq
\frac{u c(u)}{\beta(u)}
\leq
f_2(u),
\end{equation}
where 
\begin{equation*}
f_1(u):=
\frac{
u r^4
\arccos \left(
\frac{ - u^2}{1+u^2}
\right)
}{
\sqrt{1+2u^2}
\arccos \left(
\frac{-1 + u^2}{1+u^2}
\right)
}, 
\quad
f_2(u):=
\frac{
u \arccos \left(
\frac{ - u^2}{1+u^2}
\right)
}{
r^4
\sqrt{1+2u^2}
\arccos \left(
\frac{-1 + u^2}{1+u^2}
\right)
}.
\end{equation*}
Since $\lim_{u \rightarrow \infty} f_1(u)= \infty$ and 
$\lim_{u \rightarrow 0+} f_2(u)= 0$, 
there exist positive constants $m$ and $M$ depend only on $r$ such that
$0<m <M$, $f_1(M) \geq 1$, and $f_2(m) \leq 1/4$.
From \eqref{eq:30}, we have
\begin{equation*}
\frac{m c(m)}{\beta(m)} \leq f_2(m) \leq \frac{1}{4}, \quad
\frac{M c(M)}{\beta(M)} \geq f_1(M) \geq 1.
\end{equation*}
By the intermediate value theorem, there exists $u \in (m, M)$ such that
$u c(u)/\beta(u)=1/2$,
which completes the proof.
\end{proof}

\subsection{Proof of the inequality}
\label{sec:4.2}

\begin{proposition}
\label{prop:3}
For any $K \in \checkK(S_4)$, the inequality in Theorem \ref{thm:2} holds.
\end{proposition}

\begin{proof}
Let $K \in \checkK(S_4)$.
By Proposition \ref{prop:4}, there exists $u>0$ such that the condition \eqref{eq:4} holds.
Since \eqref{eq:4} is invariant under dilations of $K$,
we may assume that $\vv_0 \in \partial K$.
Then, by the $S_4$-symmetry of $K$, $\vv_k \in  \partial K$ holds for $k \in \Z$.
We put 
\begin{equation*}
\begin{pmatrix}
s^\circ \\ t^\circ \\ u^\circ
\end{pmatrix}
:=
\Lambda \vv_0 \ \in \partial K^\circ.
\end{equation*}
Since $\vv_0 \cdot \Lambda \vv_0=1$, we have $s^\circ=1-uu^\circ$.
It follows from $\Lambda(\vv_k)=\Lambda(g^k \vv_0)=g^k \Lambda(\vv_0)$ $(k \in \Z)$ that
\begin{equation*}
\Lambda \vv_0
=
\begin{pmatrix}
1-uu^\circ \\ t^\circ \\ u^\circ
\end{pmatrix}, \quad
\Lambda \vv_1
=
\begin{pmatrix}
-t^\circ \\  1-uu^\circ \\ -u^\circ
\end{pmatrix}, \quad
\Lambda \vv_2
=
\begin{pmatrix}
-(1-uu^\circ) \\ -t^\circ \\ u^\circ
\end{pmatrix}.
\end{equation*}
On the other hand, 
by \eqref{eq:50}, \eqref{eq:51}, and \eqref{eq:4}, we have
\begin{equation*}
\begin{aligned}
\frac{9}{16} \abs{K} \abs{K^\circ} \geq 
(*) 
&=
(2uc+\beta)(-a^\circ+b^\circ)+(2uc+\beta)\beta^\circ
+4c c^\circ \\
&=
4uc(-a^\circ+b^\circ)+2\beta \beta^\circ
+4c c^\circ \\
&=
4c(-u a^\circ+u b^\circ + c^\circ) + 2 \beta \beta^\circ
\\ 
&=
4
\overline{\mathcal{C}(\vv_0, \vv_1)}
\cdot
\overline{\Lambda (\mathcal{C}(\vv_0, \vv_1))}
+
2
\overline{\mathcal{C}(\vv_2, \vv_0)}
\cdot
\overline{\Lambda (\mathcal{C}(\vv_2, \vv_0))}.
\end{aligned}
\end{equation*}
Applying Lemma \ref{lem:2} to the last line, we get
\begin{equation*}
\begin{aligned}
\frac{9}{16} \abs{K} \abs{K^\circ} \geq 
(*) & \geq
(\vv_0-\vv_1)\cdot(\Lambda \vv_0 - \Lambda \vv_1) 
+ \frac{1}{2}
(\vv_2-\vv_0)\cdot(\Lambda \vv_2 - \Lambda \vv_0)  \\
&=
\begin{pmatrix}
1 \\ -1 \\ 2u
\end{pmatrix}
\cdot
\begin{pmatrix}
1-uu^\circ+t^\circ \\ t^\circ-(1-uu^\circ) \\ 2u^\circ
\end{pmatrix}
+\frac{1}{2}
\begin{pmatrix}
-2 \\ 0 \\ 0
\end{pmatrix}
\cdot
\begin{pmatrix}
-2 (1-uu^\circ) \\ -2t^\circ \\ 0
\end{pmatrix}
=4.
\end{aligned}
\end{equation*}
Consequently, we obtain
\begin{equation*}
\abs{K} \, \abs{K^\circ} \geq \frac{16}{9} (*) \geq \frac{64}{9}.
\end{equation*}
\end{proof}

\subsection{Equality condition}
\label{sec:4.3}

\begin{proof}[Proof of Theorem \ref{thm:2} (equality condition)]
Let $K \in \mathcal{K}(S_4)$ and asuume that the equality condition $\abs{K} \abs{K^\circ}=64/9$ holds.
By Proposition \ref{prop:5}, 
there exists a sequence of convex bodies $\{K^{(n)}\}_{n \in \N} \subset \checkK(S_4)$ such that $\lim_{n \rightarrow \infty} \delta(K^{(n)}, K) = 0$.
As in the proof of the inequality, for each $K^{(n)}$, we put
\begin{equation*}
\vv^{(n)}_0:=
\begin{pmatrix}
1 \\ 0 \\ u_n
\end{pmatrix}, \quad
\vv^{(n)}_k:= g^k \vv^{(n)}_0
\quad
(k \in \Z),
\end{equation*}
where each constant $u_n>0$ satisfies the condition
\begin{equation}
\label{eq:22}
|o*\mathcal{C}_{K^{(n)}}(\vv^{(n)}_2,\vv^{(n)}_0)|
=
\frac{2 u_n |o*\mathcal{C}_{K^{(n)}}(\vv^{(n)}_0,\vv^{(n)}_1)|}{\sqrt{1+2 u_n^2}},
\end{equation}
which is guaranteed by Proposition \ref{prop:4}.
Since $K^{(n)} \rightarrow K$ as $n \rightarrow \infty$ in the Hausdorff distance, 
there exists $r>0$ independent of $n$ such that $rB \subset K^{(n)} \subset r^{-1}B$ $(n \in \N)$. As we have seen in the proof of Proposition \ref{prop:4}, there exist positive numbers $m$ and $M$ depend only on $r$, and $u_n \in (m,M)$ holds for any $n \in \N$. Thus, by passing to a subsequence if necessary, there exists $u \in [m,M]$ such that $u_n \rightarrow u$ $(n \rightarrow \infty)$.
By using this $u>0$, we put
\begin{equation*}
\vv_0:=
\begin{pmatrix}
1 \\ 0 \\ u
\end{pmatrix}, \quad
\vv_k:= g^k \vv_0
\quad
(k \in \Z).
\end{equation*}
By dilations of $K$ and $K^{(n)}$, we may assume that $\vv_0 \in \partial K$ and 
$\vv^{(n)}_0 \in \partial K^{(n)}$ $(n \in \N)$.
Moreover, by passing to a subsequence if necessary, 
there exists $\vv^\circ_0 \in \partial K^\circ$ such that
\begin{equation*}
\Lambda_{K^{(n)}}\vv^{(n)}_{0} \rightarrow \vv^\circ_0 \quad (n \rightarrow \infty).
\end{equation*}
We put $\vv^\circ_k := g^k \vv^\circ_0$ $(k \in \Z)$, then 
\begin{equation*}
\vv_k \cdot \vv^\circ_k=1 \quad (k \in \Z)
\end{equation*}
hold. Since $\vv_0 \cdot \vv^\circ_0=1$, 
there exist $t^\circ, u^\circ \in \R$ such that
\begin{equation*}
\vv^\circ_0=
\begin{pmatrix}
1- u u^\circ \\ t^\circ \\ u^\circ
\end{pmatrix}.
\end{equation*}
By definition, we have $\vv_k \in K$ and $\vv^\circ_k \in K^\circ$, which imply
\begin{equation*}
\vv_1 \cdot \vv^\circ_0 = t^\circ-uu^\circ \leq 1, \quad
\vv_2 \cdot \vv^\circ_0 = 2uu^\circ-1 \leq 1, \quad
\vv_3 \cdot \vv^\circ_0 = -t^\circ-uu^\circ \leq 1,
\end{equation*}
hence
\begin{equation}
\label{eq:12}
\abs{t^\circ} \leq 1+uu^\circ, \quad \abs{uu^\circ}\leq 1.
\end{equation}

Next, we put
\begin{equation*}
\begin{aligned}
H^{(n)}_k &:= \Span \{\vv^{(n)}_{k}, \vv^{(n)}_{k+1}\}, \qquad
L^{(n)}_k := K^{(n)} \cap \pos \{\vv^{(n)}_{k}, \vv^{(n)}_{k+1}\} \subset H^{(n)}_k, \\
(L^{(n)}_k)^\circ &:= 
o*\mathcal{C}_{(K^{(n)})^\circ}(H^{(n)}_k; \Lambda_{K^{(n)}} \vv^{(n)}_{k},  \Lambda_{K^{(n)}} \vv^{(n)}_{k+1}), \\
H^{(n)} &:= \Span \{\vv^{(n)}_2, \vv^{(n)}_0\}, \qquad
L^{(n)} := K^{(n)} \cap \pos \{\vv^{(n)}_2, \vv^{(n)}_0\} \subset H^{(n)}, \\
(L^{(n)})^\circ &:= 
o*\mathcal{C}_{(K^{(n)})^\circ}(H^{(n)}; \Lambda_{K^{(n)}} \vv^{(n)}_2,  \Lambda_{K^{(n)}} \vv^{(n)}_0) \\
\end{aligned}
\end{equation*}
for each $n \in \N$.
Similarly, we put
\begin{equation*}
\begin{aligned}
H_k &:= \Span \{\vv_{k}, \vv_{k+1}\}, &
L_k &:= K \cap \pos \{\vv_{k}, \vv_{k+1}\}, &
L^\circ_k &:= 
o* \mathcal{C}_{K^\circ}(H_k; \vv^\circ_{k},  \vv^\circ_{k+1}), \\
H &:= \Span \{\vv_2, \vv_0\}, &
L &:= K \cap \pos \{\vv_2, \vv_0\}, &
L^\circ &:= 
o* \mathcal{C}_{K^\circ}(H; \vv^\circ_2,  \vv^\circ_0).
\end{aligned}
\end{equation*}
Owing to \eqref{eq:22},
by the same way as in the proof of the inequality, we obtain
\begin{equation*}
\begin{aligned}
&\abs{K^{(n)}} \abs{(K^{(n)})^\circ} \\
&\geq 
\frac{16}{9}
\left(
4
\overline{\mathcal{C}_{K^{(n)}}(\vv^{(n)}_0, \vv^{(n)}_1)}
\cdot
\overline{\Lambda_{K^{(n)}} (\mathcal{C}_{K^{(n)}}(\vv^{(n)}_0, \vv^{(n)}_1))}
+
2
\overline{\mathcal{C}_{K^{(n)}}(\vv^{(n)}_2, \vv^{(n)}_0)}
\cdot
\overline{\Lambda_{K^{(n)}} (\mathcal{C}_{K^{(n)}}(\vv^{(n)}_2, \vv^{(n)}_0))}
\right) \\
&=
\frac{16}{9}
\left(
4
\abs{L^{(n)}_0}
\abs{(L^{(n)}_0)^\circ}
+
2
\abs{L^{(n)}}
\abs{(L^{(n)})^\circ}
\right) \\
&\geq
\frac{16}{9}
\left(
(\vv^{(n)}_0-\vv^{(n)}_1)\cdot
(\Lambda_{K^{(n)}}\vv^{(n)}_0-\Lambda_{K^{(n)}}\vv^{(n)}_1)
+
\frac{1}{2}
(\vv^{(n)}_2-\vv^{(n)}_0)\cdot
(\Lambda_{K^{(n)}}\vv^{(n)}_2-\Lambda_{K^{(n)}}\vv^{(n)}_0)
\right)
=\frac{64}{9}.
\end{aligned}
\end{equation*}
As $n\rightarrow\infty$, the equality condition for $K$ asserts that
\begin{equation}
\label{eq:23}
\abs{L_{0}}
\abs{L^\circ_{0}}
=\frac{1}{4}
(\vv_0-\vv_1)\cdot
(\vv^\circ_0-\vv^\circ_1), \quad
\abs{L}
\abs{L^\circ}
=\frac{1}{4}
(\vv_2-\vv_0)\cdot
(\vv^\circ_2-\vv^\circ_0).
\end{equation}

Now, we check the assumptions of Lemma \ref{lem:4} for $L$ and $L^\circ$.
Since
\begin{equation*}
\vv_0 \cdot \vv^\circ_0 =
\vv_2 \cdot \vv^\circ_2 =1, \quad
\vv_2 \cdot \vv^\circ_0 =
\vv_0 \cdot \vv^\circ_2 =-1+2u u^\circ,
\end{equation*}
we can apply Lemma \ref{lem:4} to the pair $L$ and $L^\circ$ if $uu^\circ<1$.
In addition, the equality condition holds from \eqref{eq:23}.
Thus, we can characterize $L$ and $L^\circ$ as follows:
\begin{equation}
\label{eq:17}
\begin{aligned}
&L=o*([\vv_2,\vw] \cup [\vw,\vv_0]), \quad
L^\circ= 
o*([\Pi_H \vv^\circ_2, \vw^\circ] \cup [\vw^\circ, \Pi_H \vv^\circ_0]), \\
&\vw:=\delta \vn \times (\vv^\circ_2-\vv^\circ_0), \quad
\vw^\circ:=\delta^\circ \vn \times (\vv_2-\vv_0), \quad
\vn:=\frac{\vv_2 \times \vv_0}{\norm{\vv_2 \times \vv_0}}=
\begin{pmatrix}
0 \\ 1 \\ 0
\end{pmatrix},
\end{aligned}
\end{equation}
if $uu^\circ<1$,
where $\delta$ and $\delta^\circ$ are constants which satisfy either of the following:
\begin{equation*}
\begin{aligned}
\text{Case A: } &&
\delta&=\frac{1}{\det(\vn, \vv^\circ_2, \vv^\circ_0)}=\frac{1}{2 u^\circ (1-u u^\circ)}, \quad
\delta^\circ=\frac{\det(\vn, \vv^\circ_2, \vv^\circ_0)}{(\vv_2-\vv_0)\cdot(\vv^\circ_2-\vv^\circ_0)}=\frac{u^\circ}{2}, \\
\text{Case B: } &&
\delta&=\frac{\det(\vn, \vv_2 , \vv_0)}{(\vv_2-\vv_0)\cdot(\vv^\circ_2-\vv^\circ_0)}=\frac{u}{2(1-u u^\circ)}, \quad
\delta^\circ=\frac{1}{\det(\vn, \vv_2 , \vv_0)}=\frac{1}{2u}.
\end{aligned}
\end{equation*}
Moreover, Case A can occur only when $\angle (\Pi_H \vv^\circ_2) o (\Pi_H \vv^\circ_0) < \pi$ holds on $H$, which is equivalent to 
$(\vn \times \Pi_H \vv^\circ_2) \cdot (\Pi_H \vv^\circ_0)=
(\vn \times \vv^\circ_2) \cdot \vv^\circ_0 = 2u^\circ (1-uu^\circ)>0$,
that is, $u^\circ>0$ holds.
In addition, by the definition of $L^\circ$, there exists $\sigma^\circ \in \R$ such that $\hat{\vw}^\circ:=\vw^\circ+\sigma^\circ \vn \in K^\circ$.
By the $S_4$-symmetry and the convexity of $K^\circ$, we see
\begin{equation*}
g^2 \hat{\vw}^\circ = \vw^\circ - \sigma^\circ \vn \in K^\circ, \quad
\vw^\circ = \frac{\hat{\vw}^\circ + g^2 \hat{\vw}^\circ}{2}\in K^\circ.
\end{equation*}

Next, we check the assumptions of Lemma \ref{lem:4} for $L_{0}$ and $L^\circ_{0}$.
Since
\begin{equation*}
\vv_0 \cdot \vv^\circ_0=
\vv_1 \cdot \vv^\circ_1=1, \quad
\vv_0 \cdot \vv^\circ_1=-t^\circ-uu^\circ, \quad
\vv_1 \cdot \vv^\circ_0=t^\circ-uu^\circ,
\end{equation*}
we can apply Lemma \ref{lem:4} to the pair $L_0$ and $L^\circ_0$
if $uu^\circ>-1$ and $\abs{t^\circ}<1+u u^\circ$.
From \eqref{eq:23}, 
there exists $\sigma^\circ_0 \in \R$ such that
\begin{equation}
\label{eq:21}
\begin{aligned}
&L_0=
o*([\vv_0,\vw_0] \cup [\vw_0,\vv_1]), \quad
L^\circ_0= 
o*([\Pi_{H_0} \vv_{0}^\circ, \tilde{\vw}^\circ_0] \cup [\tilde{\vw}^\circ_0, \Pi_{H_0} \vv_1^\circ]), \\
&\vw_0:=\delta_0 \vn_0 \times (\vv^\circ_0-\vv^\circ_1), \quad
\tilde{\vw}^\circ_0:=\delta^\circ_0 \vn_0 \times (\vv_0-\vv_1), \\
& \tilde{\vw}^\circ_0=\Pi_{H_0} \vw^\circ_0, \quad
\vw^\circ_0:=\delta^\circ_0 \vn_0 \times (\vv_0-\vv_1) + \sigma^\circ_0 \vn_0 \in K^\circ, \\
&\vn_0:=\frac{\vv_0 \times \vv_1}{\norm{\vv_0 \times \vv_1}}=
\frac{1}{\sqrt{1+2u^2}}
\begin{pmatrix}
-u \\ u \\ 1
\end{pmatrix},
\end{aligned}
\end{equation}
if $uu^\circ>-1$ and $\abs{t^\circ}<1+u u^\circ$,
where $\delta_0$ and $\delta^\circ_0$ satisfy either
\begin{equation*}
\begin{aligned}
&\text{Case A$'$:} &
\delta_0&=\frac{1}{\det(\vn_0, \vv^\circ_0,  \vv^\circ_1)}=\frac{\sqrt{1 +2 u^2}}{1-(u u^\circ)^2 +(t^\circ)^2}, \\
&& \delta^\circ_0&=\frac{\det(\vn_0, \vv^\circ_0,  \vv^\circ_1)}{(\vv_0-\vv_1)\cdot(\vv^\circ_0-\vv^\circ_1)}
=\frac{1-(u u^\circ)^2 +(t^\circ)^2}{2(1+uu^\circ)\sqrt{1+2u^2}}, \\
&\text{or} && \\
&\text{Case B$'$:}&
\delta_0&=\frac{\det(\vn_0, \vv_0,  \vv_1)}{(\vv_0-\vv_1)\cdot(\vv^\circ_0-\vv^\circ_1)}=\frac{\sqrt{1+2u^2}}{2(1+uu^\circ)}, \quad
\delta^\circ_0=\frac{1}{\det(\vn_0, \vv_0,  \vv_1)}=\frac{1}{\sqrt{1+2u^2}}.
\end{aligned}
\end{equation*}
Moreover, Case A$'$ can occur only when $\angle (\Pi_{H_0}\vv^\circ_{0})o (\Pi_{H_0}\vv^\circ_1) < \pi$ holds on $H_0$, which is equivalent to 
$(\vn_0 \times \Pi_{H_0} \vv^\circ_{0}) \cdot (\Pi_{H_0}\vv^\circ_1)=(\vn_0 \times \vv^\circ_{0}) \cdot \vv^\circ_1>0$. 
Thus Case A$'$ can occur only if 
\begin{equation}
\label{eq:25}
1- (u u^\circ)^2 + (t^\circ)^2>0.
\end{equation}

To continue the argument, based on the above calculations and \eqref{eq:12}, 
we shall consider the following five cases:
\begin{description}
\item[Case I] $-1 < u u^\circ < 1$, $\abs{t^\circ} < 1+u u^\circ$.
\item[Case II] $-1 < u u^\circ < 1$, $t^\circ = \pm(1+u u^\circ)$.
\item[Case III] $u u^\circ=-1$. 
\item[Case IV] $u u^\circ=1$, $\abs{t^\circ}<2$.
\item[Case V] $u u^\circ=1$, $t^\circ=\pm 2$.
\end{description}

\paragraph{Case I: $-1 < u u^\circ < 1$, $\abs{t^\circ} < 1+u u^\circ$.}

In this case, we can apply Lemma \ref{lem:4} to the pairs $L$ and $L^\circ$, $L_0$ and $L^\circ_0$.
Then, as we have seen,
\eqref{eq:17} and \eqref{eq:21} hold. 
By using $\vw_1:=g \vw_0$ and $\vw^\circ_1:=g \vw^\circ_0$, 
we have
\begin{equation*}
\begin{aligned}
\mathcal{C}&:= 
\mathcal{C}_{K}(\vv_0, \vw_0, \vv_1, \vw_1, \vv_2, \vw, \vv_0) \\
&=
[\vv_0, \vw_0] \cup
[\vw_0, \vv_1] \cup
[\vv_1, \vw_1] \cup
[\vw_1, \vv_2] \cup
[\vv_2, \vw] \cup
[\vw, \vv_0], \\
\mathcal{C^\circ}&:= 
\mathcal{C}_{K^\circ}(
\vv^\circ_0, \vw^\circ_0, \vv^\circ_1, \vw^\circ_1, \vv^\circ_2, \vw^\circ, \vv^\circ_0) \\
&=
[\vv^\circ_0, \vw^\circ_0] \cup
[\vw^\circ_0, \vv^\circ_1] \cup
[\vv^\circ_1, \vw^\circ_1] \cup
[\vw^\circ_1, \vv^\circ_2] \cup
[\vv^\circ_2, \vw^\circ] \cup
[\vw^\circ, \vv^\circ_0],
\end{aligned}
\end{equation*}
as in Section \ref{sec:3.4}.
Here we make some calculations which will be used later:
\begin{equation*}
\begin{aligned}
\vw_0&=\frac{\delta_0}{\sqrt{1+2u^2}}
\begin{pmatrix}
1+uu^\circ-t^\circ \\
1+uu^\circ+t^\circ \\
-2u t^\circ
\end{pmatrix}, \quad
\vw^\circ_0=
\delta^\circ_0 \sqrt{1+2u^2}
\begin{pmatrix}
1 \\ 1 \\ 0
\end{pmatrix}
+ 
\frac{\sigma^\circ_0}{\sqrt{1+2u^2}}
\begin{pmatrix}
-u \\ u \\ 1
\end{pmatrix}, \\
\vw&=2 \delta
\begin{pmatrix}
0 \\ 0 \\ 1-uu^\circ
\end{pmatrix}, \quad
\vw^\circ =
2 \delta^\circ
\begin{pmatrix}
0 \\ 0 \\ 1
\end{pmatrix}, 
\end{aligned}
\end{equation*}
\begin{equation*}
\begin{aligned}
& \vv_0 \times \vw_0 + \vw_0 \times \vv_1
=\frac{2 \delta_0 (1+uu^\circ)}{\sqrt{1+2u^2}}
\begin{pmatrix}
-u \\ u \\ 1
\end{pmatrix}, \\
&\vv^\circ_0 \times \vw^\circ_0 + \vw^\circ_0 \times \vv^\circ_1
=2 \delta^\circ_0 \sqrt{1+2u^2}
\begin{pmatrix}
-u^\circ \\ u^\circ \\ 1-uu^\circ
\end{pmatrix}
+\frac{\sigma^\circ_0}{\sqrt{1+2u^2}}
\begin{pmatrix}
-(1+ uu^\circ - t^\circ) \\
-(1+ uu^\circ + t^\circ) \\
2u t^\circ
\end{pmatrix}
, \\
& \vv_1 \times \vw_1 + \vw_1 \times \vv_2
=\frac{2 \delta_0 (1+uu^\circ)}{\sqrt{1+2u^2}}
\begin{pmatrix}
u \\ u \\ 1
\end{pmatrix}, \\
&\vv^\circ_1 \times \vw^\circ_1 + \vw^\circ_1 \times \vv^\circ_2
=2 \delta^\circ_0 \sqrt{1+2u^2}
\begin{pmatrix}
u^\circ \\ u^\circ \\ 1-uu^\circ
\end{pmatrix}
+\frac{\sigma^\circ_0}{\sqrt{1+2u^2}}
\begin{pmatrix}
-(1+ uu^\circ + t^\circ) \\
1+ uu^\circ - t^\circ \\
2u t^\circ
\end{pmatrix}, \\
& \vv_2 \times \vw + \vw \times \vv_0
=
4 \delta
\begin{pmatrix}
0 \\ 1-u u^\circ \\ 0
\end{pmatrix}, \quad
\vv^\circ_2 \times \vw^\circ + \vw^\circ \times \vv^\circ_0
=
4 \delta^\circ
\begin{pmatrix}
-t^\circ \\ 1-u u^\circ \\  0
\end{pmatrix}.
\end{aligned}
\end{equation*}
Then, we have
\begin{align}
\label{eq:31}
\abs{o*\mathcal{C}_K(\vv_2, \vv_0)}&=\frac{\norm{\vv_2 \times \vw + \vw \times \vv_0}}{2}
=2 \delta (1-u u^\circ)
=
\begin{cases}
1/u^\circ & (\text{Case A}), \\
u & (\text{Case B}),
\end{cases}
\\
\notag
\frac{2 u \abs{o*\mathcal{C}_K(\vv_{0}, \vv_{1})}}{\sqrt{1+2u^2}}
&=
\frac{u \norm{\vv_{0} \times \vw_0 + \vw_0 \times \vv_{1}}}{\sqrt{1+2u^2}}
=
\frac{2 \delta_0 u (1+uu^\circ)}{\sqrt{1+2u^2}} \\
\label{eq:43}
&=
\begin{cases}
2 u (1+uu^\circ)/ (1-(u u^\circ)^2 +(t^\circ)^2) & (\text{Case A$'$}), \\
u & (\text{Case B$'$}).
\end{cases}
\end{align}
In addition, as the limit of \eqref{eq:22} as $n\rightarrow \infty$,
the condition \eqref{eq:4} for $K$ holds, which implies
\begin{equation}
\label{eq:44}
2 \delta (1-u u^\circ)
=
\frac{2 \delta_0 u (1+uu^\circ)}{\sqrt{1+2u^2}}.
\end{equation}

Here, for $\vx:=(-t^\circ, 1+u u^\circ, 2 u (1-u u^\circ))^\mathrm{T} \in \R^3$, 
we can easily check that $\mathcal{C}$ is a simple closed curve around $\vx/\mu_K(\vx)$ on $\partial K$.
In contrast, the shape of the curve $\mathcal{C}^\circ$ is slightly complicated.
We now check that $\mathcal{C}^\circ$ is a simple closed curve around $\vx^\circ/\mu_{K^\circ}(\vx^\circ)$ on $\partial K^\circ$, where $\vx^\circ:=(0,2u,1)^\mathrm{T}$.
By 
\begin{equation*}
\delta \delta^\circ= \frac{1}{4(1-uu^\circ)}, \quad
\delta_0 \delta^\circ_0 = \frac{1}{2(1+uu^\circ)}, \quad
\vw \cdot \vw^\circ_0 \leq 1, \quad
(g\vw) \cdot \vw^\circ_0=-\vw \cdot \vw^\circ_0 \leq 1,
\end{equation*}
and \eqref{eq:44},  
we have
\begin{equation}
\label{eq:45}
\abs{\sigma^\circ_0} u \leq \delta^\circ_0 (1+2 u^2).
\end{equation}
From \eqref{eq:12} and \eqref{eq:45}, we obtain 
\begin{equation*}
(\vv^\circ_0 \times \vw^\circ_0) \cdot \vx^\circ = \frac{(1+uu^\circ-t^\circ)(\delta^\circ_0(1+2u^2) - \sigma^\circ_0 u)}{\sqrt{1+2u^2}} \geq 0.
\end{equation*}
Similary, by a direct calculation, we see that
$(\vw^\circ_0 \times \vv^\circ_1)\cdot \vx^\circ$,
$(\vv^\circ_1 \times \vw^\circ_1)\cdot \vx^\circ$,
$(\vw^\circ_1 \times \vv^\circ_2)\cdot \vx^\circ$,
$(\vv^\circ_2 \times \vw^\circ)\cdot \vx^\circ$, and
$(\vw^\circ \times \vv^\circ_0)\cdot \vx^\circ$ are non-negative. Hence, 
$\mathcal{C}^\circ$ is a simple closed curve around $\vx^\circ/\mu_{K^\circ}(\vx^\circ)$ on $\partial K^\circ$.

We put $K_0 := o*\mathcal{S}_K(\mathcal{C})$ and $K^\circ_0 := o*\mathcal{S}_{K^\circ}(\mathcal{C}^\circ)$. Then, by the $S_4$-symmetry, we obtain $\abs{K}=4 \abs{K_0}$ and $\abs{K^\circ}=4 \abs{K^\circ_0}$.
In addition, by the above calculations, we have
\begin{equation*}
\begin{aligned}
\overline{\mathcal{C}}
&= 
\frac{2 \delta_0 u(1+uu^\circ)}{\sqrt{1+2u^2}}
\begin{pmatrix}
0 \\ 1 \\ 1/u
\end{pmatrix}
+
2 \delta (1-uu^\circ)
\begin{pmatrix}
0 \\ 1 \\ 0
\end{pmatrix}, 
\\
\overline{\mathcal{C}^\circ}
&= 
2 \delta^\circ_0 \sqrt{1+2u^2}
\begin{pmatrix}
0 \\ u^\circ  \\ 1-uu^\circ
\end{pmatrix}
+
2 \delta^\circ 
\begin{pmatrix}
-t^\circ  \\ 1-uu^\circ \\ 0
\end{pmatrix}
+\frac{\sigma^\circ_0}{\sqrt{1+2u^2}}
\begin{pmatrix}
-(1+ uu^\circ)\\
- t^\circ \\
2u t^\circ
\end{pmatrix}.
\end{aligned}
\end{equation*}
It follows from the equality condition, Lemma \ref{lem:7}, and \eqref{eq:44} that
\begin{equation}
\label{eq:33}
4 = 9 \abs{K_0} \abs{K^\circ_0} \geq 
\overline{\mathcal{C}} \cdot
\overline{\mathcal{C}^\circ} = 4,
\end{equation}
which means that the equality condition in Lemma \ref{lem:7} is satisfied. 
Thus, there exist $\overline{\vv}, \overline{\vv}^\circ \in \R^3$ and $\tau, \tau^\circ>0$ such that
\begin{equation}
\label{eq:34}
K_0=o*(\overline{\vv} * \mathcal{C}), \quad
K^\circ_0=o* (\overline{\vv}^\circ *\mathcal{C}^\circ), \quad
\overline{\vv} = \tau \overline{\mathcal{C}^\circ}, \quad
\overline{\vv}^\circ = \tau^\circ \overline{\mathcal{C}}.
\end{equation}
In a similar way as in Section \ref{sec:3.4}, we see that $K$ and $K^\circ$ are polyhedra.
Moreover, for each face $F$ of $K$, there exists $i \in \{0,1,2,3\}$ such that $g^i \overline{\vv} \in F$.
Similarly, for each face $F^\circ$ of $K^\circ$, there exists $i \in \{0,1,2,3\}$ such that $g^i \overline{\vv}^\circ \in F^\circ$.

Next, we carry out case analysis with respect to \eqref{eq:31} and \eqref{eq:43}.
If we suppose Case A and Case B$'$, then
\eqref{eq:31}, \eqref{eq:43}, and \eqref{eq:44} assert $1/u^\circ=u$, which contradicts $u u^\circ \in (-1,1)$.
If we suppose Case B and Case A$'$, then 
$(1+u u^\circ)^2 =(t^\circ)^2$ holds in the same way, which contradicts $\abs{t^\circ}<1+uu^\circ$. 

We now assume Case A and Case A$'$. 
Recall that $\vw_0$, $\vw_1$, $\vw$ are on $\mathcal{C}_K(\vv_0, \vv_1)$, $\mathcal{C}_K(\vv_1, \vv_2)$, $\mathcal{C}_K(\vv_2, \vv_0)$, respectively, and $\det(\vv_0, \vv_1, \vv_2) =2u>0$.
Thus, we have
\begin{equation}
\label{eq:13}
\mathcal{S}_K(\vw, \vv_0, \vw_0), \quad
\mathcal{S}_K(\vw_0, \vv_1, \vw_1), \quad
\mathcal{S}_K(\vw_1, \vv_2, \vw) \subset \mathcal{S}_K(\mathcal{C}).
\end{equation}
On the other hand, 
by \eqref{eq:31}, \eqref{eq:43}, and \eqref{eq:44}, we have
\begin{equation}
\label{eq:41}
\frac{2 \delta_0 u (1+uu^\circ)}{\sqrt{1+2u^2}} =
\frac{2 u (1+uu^\circ)}{1-(u u^\circ)^2 +(t^\circ)^2}
=
\frac{1}{u^\circ}, \quad
\frac{\delta_0(1-(u u^\circ)^2 +(t^\circ)^2)}{\sqrt{1+2u^2}} =
1,
\end{equation}
from which we obtain
$\vw \cdot \vv^\circ_0 = \vv_0 \cdot \vv^\circ_0 = \vw_0 \cdot \vv^\circ_0=1$ and 
$\det(\vw, \vv_0, \vw_0)=2 \delta \delta_0 (1-u u^\circ)(1+u u^\circ+t^\circ)/\sqrt{1+2u^2}>0$.
Thus there exists a face $F$ of $K$ such that 
$\mathcal{S}_K(\vw, \vv_0, \vw_0) = \conv\{\vw, \vv_0, \vw_0\} \subset F$ and the polar dual $\vv^\circ_0$ of $F$ is a vertex of $K^\circ$.
From \eqref{eq:34}, we obtain $\overline{\vv} * \mathcal{C} = \mathcal{S}_K(\mathcal{C})$, which implies $\overline{\vv} \in F$ by \eqref{eq:13}.
Hence, we see that $\overline{\vv} \cdot \vv^\circ_0=1$.
Similarly, we also have $\overline{\vv} \cdot \vv^\circ_1=\overline{\vv} \cdot \vv^\circ_2=1$.
Since 
\begin{equation*}
\overline{\vv}= \tau \overline{\mathcal{C}^\circ}=
\tau
\begin{pmatrix}
-t^\circ u^\circ \\
u^\circ(1+u u^\circ) \\
2 u u^\circ(1-uu^\circ)
\end{pmatrix}
+
\frac{\tau \sigma^\circ_0}{\sqrt{1+2u^2}}
\begin{pmatrix}
-(1+uu^\circ) \\
-t^\circ \\
2ut^\circ
\end{pmatrix},
\end{equation*}
we have
\begin{equation*}
\begin{aligned}
\overline{\vv} \cdot \vv^\circ_0 &=
2 \tau u (u^\circ)^2(1-uu^\circ+t^\circ) + 
\frac{2\tau \sigma^\circ_0 u u^\circ(-1-uu^\circ+t^\circ)}{\sqrt{1+2u^2}}=1, \\
\overline{\vv} \cdot \vv^\circ_1 &=
4\tau u^2 (u^\circ)^3=1, \\ 
\overline{\vv} \cdot \vv^\circ_2 &=
2 \tau u (u^\circ)^2(1-uu^\circ-t^\circ) + 
\frac{2\tau \sigma^\circ_0 u u^\circ(1+uu^\circ+t^\circ)}{\sqrt{1+2u^2}}=1, 
\end{aligned}
\end{equation*}
which imply that
\begin{equation*}
\tau = \frac{1}{4 u^2 (u^\circ)^3}, \quad
\frac{\sigma^\circ_0}{\sqrt{1+2u^2}}
=
\frac{u^\circ t^\circ}{1+u u^\circ}.
\end{equation*}
In addition, from \eqref{eq:41} we have $(t^\circ)^2=3 (u u^\circ)^2+2uu^\circ-1=(3 uu^\circ-1)(1+uu^\circ)$.
By a direct calculation, we obtain
\begin{equation*}
\begin{aligned}
\abs{K} &\geq \abs{\conv\{\overline{\vv}, g\overline{\vv}, g^2\overline{\vv}, g^3\overline{\vv}\}} \\
&= 
\frac{4}{3} \tau^3 
\left((t^\circ)^2 + (1+uu^\circ)^2\right)
\left(
(u^\circ)^2 + \frac{(\sigma^\circ_0)^2}{1+2u^2}
\right)
\left|
2u u^\circ(1-uu^\circ) + \frac{2 \sigma^\circ_0 u t^\circ}{\sqrt{1+2u^2}}
\right|
=
\frac{4}{3 u^2 (u^\circ)^3},\\
\abs{K^\circ} &\geq \abs{\conv\{\vv^\circ_0, \vv^\circ_1, \vv^\circ_2, \vv^\circ_3\}} = 
\frac{16 u^2 (u^\circ)^3}{3}.
\end{aligned}
\end{equation*}
Here, we have used the formula
$\abs{\conv\{\vx, g\vx, g^2\vx, g^3\vx\}}=4(x^2+y^2) \abs{z}/3$  for  
$\vx=(x,y,z)^{\mathrm{T}}$,
which can be easily checked.
Therefore, by the equality condition, we get 
\begin{equation*}
K=\conv\{\overline{\vv}, g\overline{\vv}, g^2\overline{\vv}, g^3\overline{\vv}\}, \quad
K^\circ=
\conv\{\vv^\circ_0, \vv^\circ_1, \vv^\circ_2, \vv^\circ_3\},
\end{equation*}
hence, $K$ and $K^\circ$ are $3$-simplices.

Finally, we assume Case B and Case B$'$. Then, we have
\begin{equation}
\label{eq:27}
\vw^\circ=
\begin{pmatrix}
0 \\ 0 \\ 1/u
\end{pmatrix}, \quad
\vw^\circ_0=
\begin{pmatrix}
1 \\ 1 \\ 0
\end{pmatrix}
+\frac{\sigma^\circ_0}{\sqrt{1+2u^2}}
\begin{pmatrix}
-u \\ u \\ 1
\end{pmatrix}, \quad
\vw^\circ_1=
\begin{pmatrix}
-1 \\ 1 \\ 0
\end{pmatrix}
+\frac{\sigma^\circ_0}{\sqrt{1+2u^2}}
\begin{pmatrix}
-u \\ -u \\ -1
\end{pmatrix}
\end{equation}
and 
\begin{equation}
\label{eq:24}
\begin{aligned}
\vv_0 \cdot \vw^\circ &= \vv_0 \cdot \vv^\circ_0 = \vv_0 \cdot \vw^\circ_0=1, \\
\vv_1 \cdot \vw^\circ_0 &= \vv_1 \cdot \vv^\circ_1 = \vv_1 \cdot \vw^\circ_1=1, \\
\vv_2 \cdot \vw^\circ_1 &= \vv_2 \cdot \vv^\circ_2 = \vv_2 \cdot \vw^\circ=1.
\end{aligned}
\end{equation}
Moreover, by a direct calculation, it holds that
\begin{equation*}
\det(\vw^\circ, \vv^\circ_0, \vw^\circ_0)
+
\det(\vw^\circ_0, \vv^\circ_1, \vw^\circ_1)
+
\det(\vw^\circ_1, \vv^\circ_2, \vw^\circ)
=\frac{2(1+2u^2)+2(\sigma^\circ_0)^2 u^2}{u(1+2u^2)}>0,
\end{equation*}
which means that
at least one of 
the three terms of the left-hand side is not zero.
If $\det(\vw^\circ, \vv^\circ_0, \vw^\circ_0) > 0$, then, by \eqref{eq:24}, 
there exists a face $F^\circ$ of $K^\circ$ such that 
$\conv\{\vw^\circ, \vv^\circ_0, \vw^\circ_0\} \subset F^\circ$ and the polar dual $\vv_0$ of $F^\circ$ is a vertex of $K$.
By the $S_4$-symmetry, each $\vv_i$ $(i=0,1,2,3)$ is a vertex of $K$.
We have the same conclusion in the case that $\det(\vw^\circ_0, \vv^\circ_1, \vw^\circ_1) > 0$ or $\det(\vw^\circ_1, \vv^\circ_2, \vw^\circ) > 0$.
Consequently, in either case, $\vv_0$ is a vertex of $K$. Let $F^\circ$ be the polar dual of $\vv_0$.
Then, $F^\circ$ is a face of $K^\circ$ and there exists $i \in \{0,1,2,3\}$ such that
$g^i \overline{\vv}^\circ \in F^\circ$, hence
$\vv_0 \cdot g^i \overline{\vv}^\circ=1$.
On the other hand, we have
\begin{equation*}
\overline{\vv}^\circ = \tau^\circ \overline{\mathcal{C}}
=\tau^\circ
\begin{pmatrix}
0 \\ 2u \\ 1
\end{pmatrix}, \quad
\vv_0 \cdot \overline{\vv}^\circ
=
\vv_0 \cdot g^2 \overline{\vv}^\circ
=
\vv_0 \cdot g^3\overline{\vv}^\circ
=\tau^\circ u, \quad
\vv_0 \cdot g\overline{\vv}^\circ=-3 \tau^\circ u <0.
\end{equation*}
Thus, $i \in \{0,2,3\}$ and $\tau^\circ=1/u$ hold, 
hence $\overline{\vv}^\circ=(1/u) \overline{\mathcal{C}} \in K^\circ$.
Therefore, we obtain
\begin{equation*}
\begin{aligned}
\abs{K} &\geq \abs{\conv\{\vv_0, \vv_1, \vv_2, \vv_3\}} = \frac{4u}{3}, \\
\abs{K^\circ} &\geq \abs{\conv\{\overline{\vv}^\circ, g \overline{\vv}^\circ, g^2 \overline{\vv}^\circ, g^3 \overline{\vv}^\circ \}} = \frac{16}{3u},
\end{aligned}
\end{equation*}
and
\begin{equation*}
K=\conv\{\vv_0, \vv_1, \vv_2, \vv_3\}, \quad
K^\circ=\conv\{\overline{\vv}^\circ, g \overline{\vv}^\circ, g^2 \overline{\vv}^\circ, g^3 \overline{\vv}^\circ \}
\end{equation*}
from the equality condition.
Thus, $K$ and $K^\circ$ are $3$-simplices.

\paragraph{Case II: $-1 < u u^\circ < 1$, $t^\circ = \pm(1+u u^\circ)$.}

We put
\begin{equation*}
\tilde{\vv}^\circ_0:=\Pi_{H_0} \vv^\circ_0 = \vv^\circ_0 + \lambda^\circ_0 (\vv_0 \times \vv_1), \quad
\tilde{\vv}^\circ_1:=\Pi_{H_0} \vv^\circ_1 = \vv^\circ_1 + \lambda^\circ_1 (\vv_0 \times \vv_1),
\end{equation*}
where $\lambda^\circ_0$ and $\lambda^\circ_1$ are constants.
By definition, $\vv_0 \times \vv_1$ and $\tilde{\vv}^\circ_0 \times \tilde{\vv}^\circ_1$ are normal vectors of $H_0$. In addition, since
$\vv_0 \cdot (\vv_0 \times \vv_1) =\vv_1 \cdot (\vv_0 \times \vv_1) =0$ and $(t^\circ)^2 = (1+ u u^\circ)^2$, 
we get
\begin{equation*}
\begin{aligned}
(\vv_0 \times \vv_1) \cdot (\tilde{\vv}^\circ_0 \times \tilde{\vv}^\circ_1) 
&=
(\vv_0 \cdot \tilde{\vv}^\circ_0)(\vv_1 \cdot \tilde{\vv}^\circ_1)
-
(\vv_0 \cdot \tilde{\vv}^\circ_1)(\vv_1 \cdot \tilde{\vv}^\circ_0) \\
&=
(\vv_0 \cdot \vv^\circ_0)(\vv_1 \cdot \vv^\circ_1)
-
(\vv_0 \cdot \vv^\circ_1)(\vv_1 \cdot \vv^\circ_0)
=2(1+ uu^\circ) >0.
\end{aligned}
\end{equation*}
Thus, $\vv_0 \times \vv_1$ and $\tilde{\vv}^\circ_0 \times \tilde{\vv}^\circ_1$ are positively proportional. 
Hence, we obtain 
$\angle \tilde{\vv}^\circ_0 o \tilde{\vv}^\circ_1 < \pi$ and $\conv\{o,\tilde{\vv}^\circ_0, \tilde{\vv}^\circ_1\} \subset L^\circ_0$.
By the equality condition, we have
\begin{equation*}
\begin{aligned}
\frac{1}{2}(1+uu^\circ) &= \frac{1}{4}(\vv_0 - \vv_1) \cdot (\vv^\circ_0 - \vv^\circ_1) = \abs{L_0} \abs{L^\circ_0} 
 \geq 
\abs{\conv\{o,\vv_0, \vv_1\}} \abs{\conv\{o,\tilde{\vv}^\circ_0, \tilde{\vv}^\circ_1\}} \\
&= \frac{1}{4} \norm{\vv_0 \times \vv_1} \norm{\tilde{\vv}^\circ_0 \times \tilde{\vv}^\circ_1}
= \frac{1}{4} (\vv_0 \times \vv_1) \cdot (\tilde{\vv}^\circ_0 \times \tilde{\vv}^\circ_1) 
= \frac{1}{2} (1+uu^\circ).
\end{aligned}
\end{equation*}
It follows that
\begin{equation*}
L_0=
\conv\{o,\vv_0, \vv_1\}, \quad
L^\circ_0= \conv\{o,\tilde{\vv}^\circ_0, \tilde{\vv}^\circ_1\},
\end{equation*}
which mean that
\begin{equation*}
\mathcal{C}_K(\vv_0, \vv_1) = [\vv_0, \vv_1], \quad
\mathcal{C}_{K^\circ}(\vv^\circ_0, \vv^\circ_1) = [\vv^\circ_0, \vv^\circ_1],
\end{equation*}
which imply
\begin{equation}
\label{eq:32}
\frac{2 u \abs{o*C_K(\vv_{0}, \vv_{1})}}{\sqrt{1+2u^2}}
=
\frac{u \norm{\vv_0 \times \vv_1}}{\sqrt{1+2u^2}} =u.
\end{equation}
On the other hand, since $u u^\circ<1$, we can apply Lemma \ref{lem:4} to the pair $L$ and $L^\circ$, 
and so \eqref{eq:17} and \eqref{eq:31} hold.
From  \eqref{eq:4}, \eqref{eq:31}, \eqref{eq:32}, and $u u^\circ<1$, Case B must occur. Thus
\begin{equation*}
\vw=
\begin{pmatrix}
0 \\ 0 \\ u
\end{pmatrix}
\in [\vv_2, \vv_0] \subset \partial K, \quad
\vw^\circ
=
\begin{pmatrix}
0 \\ 0 \\ 1/u
\end{pmatrix}
 \in \partial K^\circ,
\end{equation*}
and so we have
\begin{equation*}
\begin{aligned}
\mathcal{C}&:= 
\mathcal{C}_{K}(\vv_{0}, \vv_1, \vv_2, \vv_0) =
[\vv_0, \vv_1] \cup
[\vv_1,\vv_2] \cup
[\vv_2, \vv_0], \\
\mathcal{C^\circ}&:= 
\mathcal{C}_{K^\circ}(\vv^\circ_0, \vv^\circ_1, \vv^\circ_2, \vw^\circ, \vv^\circ_0) =
[\vv^\circ_0, \vv^\circ_1] \cup
[\vv^\circ_1, \vv^\circ_2] \cup
[\vv^\circ_2, \vw^\circ] \cup
[\vw^\circ, \vv^\circ_0].
\end{aligned}
\end{equation*}
We put
\begin{equation*}
K_0 := o*\mathcal{S}_K(\mathcal{C}), \quad
K^\circ_0 := o*\mathcal{S}_{K^\circ}(\mathcal{C}^\circ),
\end{equation*}
then we have
\begin{equation*}
\begin{aligned}
\overline{\mathcal{C}}
&=
\frac{1}{2}\left(
\vv_0 \times  \vv_1
+
\vv_1 \times \vv_2
+
\vv_2 \times \vv_0
\right)
=
\begin{pmatrix}
0 \\ 2u \\ 1
\end{pmatrix}, \\
\overline{\mathcal{C}^\circ}
&=
\frac{1}{2}\left(
\vv^\circ_0 \times \vv^\circ_1
+
\vv^\circ_1 \times \vv^\circ_2
+
\vv^\circ_2 \times \vw^\circ
+
\vw^\circ \times \vv^\circ_0
\right)
=
\begin{pmatrix}
- t^\circ (u^\circ + 1/u) \\
(1-u u^\circ)(u^\circ+1/u) \\
(1-u u^\circ)^2 +  (t^\circ)^2
\end{pmatrix}.
\end{aligned}
\end{equation*}
Then, \eqref{eq:33} in Case I holds from the equality condtion, Lemma \ref{lem:7}, and $(t^\circ)^2=(1+uu^\circ)^2$.
Namely, the equality condition in Lemma \ref{lem:7} holds.
It follows that there exist
$\overline{\vv}, \overline{\vv}^\circ \in \R^3$ and $\tau, \tau^\circ>0$ such that \eqref{eq:34} holds.
We note that $\overline{\mathcal{C}}$ is the same one as in Case B--Case B$'$ of Case I.
Now we suppose that $t^\circ=1+uu^\circ$.
In the same way as in Case I,
we see that $\vv_2$ is a vertex of $K$, because $\vv_2 \cdot \vv^\circ_1=\vv_2 \cdot \vv^\circ_2=\vv_2 \cdot \vw^\circ=1$ and $\det(\vv^\circ_1, \vv^\circ_2, \vw^\circ) > 0$. 
Hence,
there exists $i \in \{0,1,2,3\}$ such that $\vv_2 \cdot g^i \overline{\vv}^\circ=1$.
Since
\begin{equation*}
\overline{\vv}^\circ = \tau^\circ
\begin{pmatrix}
0 \\ 2u \\ 1
\end{pmatrix}, \quad
\vv_2 \cdot \overline{\vv}^\circ
=
\vv_2 \cdot g^1 \overline{\vv}^\circ
=
\vv_2 \cdot g^2\overline{\vv}^\circ
=\tau^\circ u, \quad
\vv_2 \cdot g^3\overline{\vv}^\circ=-3 \tau^\circ u <0,
\end{equation*}
we have $i \in \{0,1,2\}$ and $\tau^\circ=1/u$,
which means $\overline{\vv}^\circ=(1/u) \overline{\mathcal{C}} \in K^\circ$.
By the same argument as in Case B--Case B$'$ of Case I, we see that $K$ and $K^\circ$ are $3$-simplices.
In the case where $t^\circ=-(1+u u^\circ)$, $\vv_0$ is a vertex of $K$, since 
$\vv_0 \cdot \vw^\circ=\vv_0 \cdot \vv^\circ_0=\vv_0 \cdot \vv^\circ_1=1$ and $\det(\vw^\circ, \vv^\circ_0, \vv^\circ_1) > 0$.
The same argument as above implies that $K$ and $K^\circ$ are $3$-simplices.

\paragraph{Case III: $u u^\circ=-1$.}

By \eqref{eq:12}, we have $t^\circ=0$.
Since $\vv^\circ_0=(2, 0, -1/u)^\mathrm{T}$, we obtain
\begin{equation*}
\abs{K} \geq \abs{\conv\{\vv_0, \vv_1, \vv_2, \vv_3\}} = \frac{4u}{3}, \quad
\abs{K^\circ} \geq \abs{\conv\{\vv^\circ_0, \vv^\circ_1, \vv^\circ_2, \vv^\circ_3\}} = \frac{16}{3u}.
\end{equation*}
By the equality condition, 
$K=\conv\{\vv_0, \vv_1, \vv_2, \vv_3\}$
and 
$K^\circ=\conv\{\vv^\circ_0, \vv^\circ_1, \vv^\circ_2, \vv^\circ_3\}$ hold, 
hence $K$ and $K^\circ$ are $3$-simplices.

\paragraph{Case IV: $u u^\circ=1$, $\abs{t^\circ}<2$.}

Since $\vv_2 \cdot \vv^\circ_2 = \vv_0 \cdot \vv^\circ_0 = \vv_2 \cdot \vv^\circ_0 = \vv_0 \cdot \vv^\circ_2=1$, the inner product of each point on $[\vv_2, \vv_0]$ and each point on $[\vv^\circ_2, \vv^\circ_0]$ is always one. This means that $[\vv_2, \vv_0] \subset \partial K$ and $[\vv^\circ_2, \vv^\circ_0] \subset \partial K^\circ$. On the other hand, since $u u^\circ>-1$ and $\abs{t^\circ}<2$, we can apply Lemma \ref{lem:4} to the pair $L_0$ and $L^\circ_0$. Then, we have
\begin{equation*}
L_0=\conv \left\{o, \vv_{0},\vw_0,\vv_1\right\}, \quad
\vw_0=\frac{\delta_0}{\sqrt{1+2u^2}}
\begin{pmatrix}
2-t^\circ \\
2+t^\circ \\
-2u t^\circ
\end{pmatrix},
\end{equation*}
since $uu^\circ=1$.
Since $\mathcal{C}_K(\vv_2, \vv_0)=[\vv_2, \vv_0]$ from $[\vv_2, \vv_0] \subset \partial K$, we obtain
\begin{equation*}
\begin{aligned}
\abs{o*C_K(\vv_2, \vv_0)}&=\frac{\norm{\vv_2 \times \vv_0}}{2}
=u, \\
\frac{2 u \abs{o*C_K(\vv_{0}, \vv_{1})}}{\sqrt{1+2u^2}}
&=
\frac{4 \delta_0 u}{\sqrt{1+2u^2}} 
=
\begin{cases}
4 u / (t^\circ)^2 & (\text{Case A$'$}), \\
u & (\text{Case B$'$}).
\end{cases}
\end{aligned}
\end{equation*}
Since \eqref{eq:4} holds for $K$, Case A$'$ does not occur.
In Case B$'$, we have
\begin{equation*}
\begin{aligned}
\mathcal{C}&:= 
\mathcal{C}_{K}(\vv_0, \vw_0, \vv_1, \vw_1, \vv_2, \vv_0) \\
&=
[\vv_0, \vw_0] \cup
[\vw_0, \vv_1] \cup
[\vv_1, \vw_1] \cup
[\vw_1, \vv_2] \cup
[\vv_2, \vv_0], \\
\mathcal{C^\circ}&:= 
\mathcal{C}_{K^\circ}(
\vv^\circ_0, \vw^\circ_0, \vv^\circ_1, \vw^\circ_1, \vv^\circ_2, \vv^\circ_0) \\
&=
[\vv^\circ_0, \vw^\circ_0] \cup
[\vw^\circ_0, \vv^\circ_1] \cup
[\vv^\circ_1, \vw^\circ_1] \cup
[\vw^\circ_1, \vv^\circ_2] \cup
[\vv^\circ_2, \vv^\circ_0], 
\end{aligned}
\end{equation*}
and put
\begin{equation*}
K_0 := o*\mathcal{S}_K(\mathcal{C}), \quad
K^\circ_0 := o*\mathcal{S}_{K^\circ}(\mathcal{C}^\circ).
\end{equation*}
Then, we have
\begin{equation*}
\overline{\mathcal{C}}
= 
\begin{pmatrix}
 0 \\ 2u \\ 1
\end{pmatrix}, \quad 
\overline{\mathcal{C}^\circ}
= 
\begin{pmatrix}
- t^\circ/u \\
2/u\\
0
\end{pmatrix}
+\frac{\sigma^\circ_0}{\sqrt{1+2u^2}}
\begin{pmatrix}
-2 \\
- t^\circ \\
2u t^\circ
\end{pmatrix},
\quad
4 = 9 \abs{K_0} \abs{K^\circ_0} \geq 
\overline{\mathcal{C}} \cdot
\overline{\mathcal{C}^\circ} = 4.
\end{equation*}
Thus, the equality condition in Lemma \ref{lem:7} holds, 
and so there exist $\overline{\vv}, \overline{\vv}^\circ \in \R^3$ and $\tau, \tau^\circ>0$ such that \eqref{eq:34} holds.
We note that $\overline{\mathcal{C}}$ is the same one as in Case B--Case B$'$ of Case I. Moreover, \eqref{eq:27} holds.
By the same argument as in Case I, we see that $K$ and $K^\circ$ are $3$-simplices.

\paragraph{Case V: $u u^\circ=1$, $t^\circ=\pm 2$.}

Since $t^\circ=\pm 2$, we have $\vv^\circ_0=(0, \pm 2, 1/u)^\mathrm{T}$, 
which asserts that
\begin{equation*}
\abs{K} \geq \abs{\conv\{\vv_0, \vv_1, \vv_2, \vv_3\}} = \frac{4u}{3}, \quad
\abs{K^\circ} \geq \abs{\conv\{\vv^\circ_0, \vv^\circ_1, \vv^\circ_2, \vv^\circ_3\}} = \frac{16}{3u}.
\end{equation*}
By the equality condition,
$K=\conv\{\vv_0, \vv_1, \vv_2, \vv_3\}$ and
$K^\circ=\conv\{\vv^\circ_0, \vv^\circ_1, \vv^\circ_2, \vv^\circ_3\}$, 
hence $K$ and $K^\circ$ are $3$-simplices.
\end{proof}

\section*{Funding}
The first author was supported by JSPS KAKENHI Grant Number JP20K03576. 
The second author was supported by JSPS KAKENHI Grant Number JP23K03189.

\end{document}